\documentclass[reqno]{amsart}

\usepackage{appendix}
\usepackage{mathtools}
\usepackage[all, cmtip]{xy}

\usepackage{amsfonts,amssymb,verbatim,amsmath,amsthm,latexsym,textcomp,amscd, hyperref}
\usepackage{epsfig}
\usepackage{amsmath,amsthm,graphics,float}
\usepackage{graphicx}
\usepackage{color}
\usepackage{url}

\theoremstyle{plain}
\newtheorem{theorem}{Theorem}[section]
\newtheorem{lemma}[theorem]{Lemma}
\newtheorem{corollary}[theorem]{Corollary}
\newtheorem*{ncorollary}{Corollary}
\newtheorem{proposition}[theorem]{Proposition}
\newtheorem{claim}{Claim}
		\newtheorem*{theorema}{Theorem A}
		\newtheorem*{theoremb}{Theorem B}
		\newtheorem*{theoremc}{Theorem C}
		\newtheorem*{nclaim}{Claim}
		\newtheorem{fact}[theorem]{Fact}
\theoremstyle{definition}
		\newtheorem{definition}[theorem]{Definition}
\theoremstyle{remark}
		\newtheorem*{remark}{Remark}
		\newtheorem{example}[theorem]{Example}
		
\newcommand{\A}{{\mathbb{A}}}

\newcommand{\C}{{\mathbb{C}}}

\newcommand{\G}{{\mathbb{G}}}
\renewcommand{\H}{{\mathbb{H}}}

\newcommand{\Q}{{\mathbb{Q}}}
\newcommand{\R}{{\mathbb{R}}}
\renewcommand{\S}{{\Sigma}}

\newcommand{\Z}{{\mathbb{Z}}}

\newcommand{\Cb}{{\mathbf{C}}}

\newcommand{\Fcal}{{\mathcal{F}}}

\newcommand{\Scal}{{\mathcal{S}}}

\DeclareMathOperator{\Aut}{Aut}

\DeclareMathOperator{\Gal}{Gal}

\DeclareMathOperator{\Hom}{Hom}

\DeclareMathOperator{\sheafhom}{\mathcal{H}\kern -.5pt \emph{om}}

\DeclareMathOperator{\Mod}{Mod}

\DeclareMathOperator{\Out}{\textup{Out}}

\DeclareMathOperator{\PGL}{PGL}
\DeclareMathOperator{\PSL}{PSL}
\DeclareMathOperator{\PSU}{PSU}

\DeclareMathOperator{\Rep}{Rep}

\DeclareMathOperator{\SL}{SL}

\DeclareMathOperator{\SO}{SO}
\DeclareMathOperator{\SU}{SU}
\DeclareMathOperator{\Sp}{Sp}

\DeclareMathOperator{\Spec}{Spec}

\DeclareMathOperator{\tw}{tw}

\DeclareMathOperator{\tr}{tr}

\newcommand{\git}{\mathbin{
  \mathchoice{/\mkern-6mu/}
    {/\mkern-6mu/}
    {/\mkern-5mu/}
    {/\mkern-5mu/}}}

\newcommand{\slc}{{\SL (2,\mathbb{C})}}

\newcommand\HHH{{\mathbb H}}

\begin{document}

\title[Surface group representations with finite  orbits]{Surface group representations
in ${\rm SL}_2({\mathbb C})$ with finite mapping class orbits}

\author[I. Biswas]{Indranil Biswas}

\address{School of Mathematics, Tata Institute of Fundamental Research, Homi 
Bhabha Road, Mumbai 400005, India}

\email{indranil@math.tifr.res.in}

\author[S. Gupta]{Subhojoy Gupta}

\address{Department of Mathematics, Indian Institute of Science,
Bangalore, India}

\email{subhojoy@gmail.com}

\author[M. Mj]{Mahan Mj}

\address{School of Mathematics, Tata Institute of Fundamental Research, Homi 
Bhabha Road, Mumbai 400005, India}

\email{mahan@math.tifr.res.in}

\author[J. P. Whang]{Junho Peter Whang}

\address{Department of Mathematics,
Massachusetts Institute of Technology, 
Office: Building 2 Room 2-238a, 77 Massachusetts Avenue
Cambridge, MA 02139-4307, USA}

\email{jwhang@mit.edu}



\date{\today}

\begin{abstract}
Given an oriented surface of positive genus with finitely many punctures, we classify the finite
orbits of the mapping class group action on the moduli space of semisimple complex special linear two dimensional 
representations of the fundamental group of the surface. For surfaces of genus at least two,
such orbits correspond to  homomorphisms with finite image. For genus one, they correspond
to the finite or special dihedral representations. We also obtain an analogous result for bounded
orbits in the moduli space.
\end{abstract}

\maketitle
\setcounter{tocdepth}{2}
\tableofcontents

\section{Introduction} \label{sect:1}

Let $\Sigma$ be an oriented surface of genus $g\geq0$ with a finite set $\Fcal$ of punctures. 
The \emph{${\rm SL}_2({\mathbb C})$-character variety} of $\Sigma$ 
$$X(\Sigma)\,=\,\Hom(\pi_1(\Sigma),\, {\rm SL}_2({\mathbb C}))\git {\rm SL}_2({\mathbb C})$$ is an 
affine algebraic variety whose complex points parametrize the conjugacy classes of semisimple 
representations $\pi_1(\Sigma)\,\to\,\SL_2(\C)$ of the fundamental group of $\Sigma$. Let 
$\Mod(\Sigma)$ denote the pure mapping class group of $\Sigma$ fixing $\Fcal$ pointwise. The 
group acts on the moduli space $X(\Sigma)$ by precomposition. This paper classifies the 
finite orbits of this action for surfaces of positive genus.

Our analysis divides into the cases of genus one and higher. For surfaces of genus at least two, we prove the following.

\begin{theorema}\label{tha}
Let $\Sigma$ be an oriented surface of genus $g\geq2$ with $n\geq 0$ punctures. A semisimple
representation $\rho\,:\,\pi_1(\Sigma)\,\to\,\SL_2(\C)$ has finite mapping class group orbit in the
character variety $X(\Sigma)$ if and only if $\rho$ is finite.
\end{theorema}

To describe the corresponding result for surfaces of genus $1$, define an irreducible 
representation $\rho:\pi_1(\Sigma)\to\SL_2(\C)$ to be \emph{special dihedral} if
\begin{itemize}
\item its image lies 
in the infinite dihedral group $D_\infty$ in $\SL_2(\C)$ (its definition is recalled in
Section \ref{sect:2}), and

\item there is a nonseparating simple closed curve $a$ in $\Sigma$ such that the restriction of 
$\rho$ to the complement $\Sigma\backslash a$ is diagonal.
\end{itemize}

\begin{theoremb}\label{thb}
Let $\Sigma$ be an oriented surface of genus $1$ with $n\geq0$ punctures. A semisimple
representation $\rho\,:\,\pi_1(\Sigma)\,\to\,\SL_2(\C)$ has finite mapping class group orbit in
$X(\Sigma)$ if and only if $\rho$ is finite or special dihedral up to conjugacy.
\end{theoremb}

Many of the technical issues in the proofs of Theorems A and B above arise while dealing with punctures. For $\Sigma$ closed (i.e. $n=0$) the proofs become considerably simplified thanks to existing results, in particular \cite{gkm,px,coopermanning}.
To illustrate this, we have included a short proof of Theorem A 
in Section \ref{sect:5} for closed surfaces of genus at least 2 using \cite{gkm,px}.

For surfaces of genus zero with more than three punctures, the description of the finite mapping 
class group orbits in the $\SL_2(\C)$-character variety is more complicated and in general unknown. Lisovyy-Tykhyy 
\cite{lt} completed the case of the four-punctured sphere as part of their classification of 
algebraic solutions to Painlev\'e VI differential equations, after earlier works including \cite{dm}, \cite{iwasaki}, \cite{hitchin}, \cite{boalch2}, \cite{cl}. In the case of the four-punctured sphere, it is known that there exist 
finite mapping class group orbits corresponding to representations with Zariski dense image in
$\SL_2(\C)$. In particular, the simplicity of results in Theorems A and B contrasts with the complexity of finite orbits in genus zero. We remark that the once-punctured torus case of Theorem B was essentially proved 
by Dubrovin-Mazzocco \cite{dm} (also in connection with Painlev\'e VI); the derivation from \cite{dm} is recorded in the Appendix. Our work gives a different proof in this case.

The relationship between algebraic solutions of Painlev\'e VI and finite mapping class group orbits in the character variety of the four-punctured sphere comes from R. Fuchs' interpretation of Painlev\'e VI \cite{fuchs} as the equation for isomonodromic deformations of Fuchsian systems on the Riemann sphere with 4 singular points. More generally, as shown in the work of Cousin \cite{cousin} and Cousin-Heu \cite{ch}, the finite mapping class group orbits in the character variety $X(\Sigma)$ correspond to algebraic isomonodromic deformations of algebraic $\SL_2$-bundles with logarithmic connection on algebraic curves. Combined with their result, the above results give a complete classification of algebraic solutions to isomonodromy equations over curves of positive genus, in the case of semisimple $\SL_2$-monodromy, in terms of monodromy data.

This paper pursues the theme of characterizing points on the character variety $X(\Sigma)$ 
with special dynamical properties. In this spirit, we also prove the result below. Given a 
complex algebraic variety $V$, we shall say that a subset of $V(\C)$ is \emph{bounded} if it 
has compact closure in $V(\C)$ with respect to the Euclidean topology.

\begin{theoremc}
Let $\Sigma$ be an oriented surface of genus $g\geq1$ with $n\geq0$ punctures. A semisimple
representation $\rho\,:\,\pi_1(\Sigma)\,\to\,\SL_2(\C)$ has bounded mapping class group
orbit in the character variety $X(\Sigma)$ if and only if
\begin{enumerate}
\item[(a)] $\rho$ is unitary up to conjugacy, or

\item[(b)] $g=1$ and $\rho$ is special dihedral up to conjugacy.
\end{enumerate} 
\end{theoremc}

Our results and methods answer some previously raised basic questions. Theorems A and B imply (Corollary 
\ref{faithful}) that a faithful representation of a positive-genus hyperbolic surface group 
into $\SL_2(\C)$ (or $\PSL_2(\C)$) cannot have finite mapping class group orbit in the 
character variety, answering a question raised by Lubotzky. Theorems A and B also verify the 
$G\,=\,\SL_2(\C)$ case of the following conjecture of Kisin \cite[Chapter 1]{sinz}: For $\pi$ the 
fundamental group of a closed surface or a free group of rank $r\geq3$, the points with finite 
orbits for the action of the outer automorphism group $\Out(\pi)$ on the character variety 
$\Hom(\pi,\,G)\git G$, for reductive algebraic groups $G$, correspond to representations $\pi\,\to\, 
G$ with virtually solvable image. However, there are counter-examples to this conjecture
for general $G$ (see, for instance \cite{KS,bkms}). Finally, we show that, given a closed hyperbolic 
surface $S$ of genus at least two, the energy of the harmonic map associated 
to a representation $\pi_1(S)\,\to\,\SL_2(\C)$ is bounded along the mapping class group orbit in the 
character variety if and only if the representation fixes a point of $\HHH^3$ (Theorem 
\ref{energy}). This answers a question due to Goldman.

For a surface $\Sigma$ of negative Euler characteristic with $n\,\geq\,1$ marked punctures, the 
subvarieties $X_k(\Sigma)$ of $X(\Sigma)$ obtained by fixing the traces 
$k\,=\,(k_1,\,\cdots,\,k_n)\,\in\,\C^n$ of local monodromy along the punctures form a family of log 
Calabi-Yau varieties \cite{whang} with rich Diophantine structure \cite{whang2}. Classifying 
the finite mapping class group orbits (and other invariant subvarieties) forms an important 
step in the study of strong approximation for these varieties, undertaken in the once-punctured 
torus case by Bourgain-Gamburd-Sarnak \cite{bgs}.

Finally, we note that Theorems A, B, C remain valid if the pure mapping class group is replaced by the full mapping class group (allowing for permutation of punctures). Similarly, since the fundamental groups of punctured surfaces are free, Theorems A and C imply the following. Given any finitely generated group $\pi$, let $X(\pi,\SL_2)$ be its $\SL_2$-character variety (see Section \ref{sect:2.2}).

\begin{ncorollary}
Let $\pi\,=\,F_m$ be a free group of rank $m\geq4$. A semisimple representation $\rho\,:\,\pi\,\to\,\SL_2(\C)$ has
finite (respectively, bounded) $\Out(\pi)$-orbit in $X(\pi,\SL_2)$ if and only if $\rho$ has finite (respectively, bounded) image.
\end{ncorollary}

\noindent{\bf Organization of the paper.}\,
In Section \ref{sect:2}, we record background on algebraic subgroups of $\SL_2(\C)$, character 
varieties, and mapping class groups. We also introduce the notion of loop configurations as a 
tool to keep track of subsurfaces of a given surface. In Section \ref{sect:3}, we study 
representations of surface groups whose images are contained in proper algebraic subgroups of 
$\SL_2(\C)$, and give a characterization of those with finite mapping class group orbits for 
surfaces of positive genus.

In Sections \ref{sect:4} and \ref{sect:5}, we prove our main results. One of the ingredients in the proof of Theorems A and B is a theorem of Patel--Shankar--Whang \cite[Theorem 1.2]{psw}, which states that a semisimple $\SL_2(\C)$-representation of a positive-genus surface group with finite monodromy along every simple loop must in fact be finite. (For the proof of Theorem C which runs in parallel, there is an analogous result.) Along essential curves, the requisite finiteness of monodromy can be largely obtained by studying Dehn twists, as described in Section \ref{sect:4}. Finiteness of local monodromy along the punctures is more involved, and is achieved in Section \ref{sect:5}. In the special case where $\Sigma$ is a closed surface, we also give a different (short) proof of Theorem A 
relying on some of the techniques developed by Gallo--Kapovich--Marden \cite{gkm} and Previte--Xia \cite{px}. 

In 
Section \ref{sect:6}, we provide applications of our work, answering earlier
mentioned questions due to Lubotzky and Goldman. Finally, in the Appendix we demonstrate a derivation of the once-punctured torus case of Theorem B from Dubrovin--Mazzocco \cite{dm}.

\section{Background} \label{sect:2}
\subsection{Subgroups of ${\rm SL}_2(\C)$} \label{sect:2.1}
Let $G$ be a proper algebraic subgroup of $\SL_2$ defined over $\C$. Up to conjugation, $G$ satisfies one of the
following \cite[Theorem 4.29]{vs}:
\begin{enumerate}
	\item[\textup{(1)}] $G$ is a subgroup of the \emph{standard Borel group}
	$$B\,=\,\left\{\begin{bmatrix}a & b\\ 0 & a^{-1}\end{bmatrix}\,\mid\, a\in\C^\times,b\in\C\right\}.$$
	\item[\textup{(2)}] $G$ is a subgroup of the \emph{infinite dihedral group}
	$$D_\infty\,=\,\left\{\begin{bmatrix}c & 0 \\ 0 & c^{-1}\end{bmatrix}\,\mid\,c\in\C^\times\right\}\cup \left\{\begin{bmatrix}0 & c \\ -c^{-1} & 0\end{bmatrix}\,\mid\,c\in\C^\times\right\}.$$
	\item[\textup{(3)}] $G$ is one of the finite groups $B A_4$, $B S_4$, and $B A_5$, which are the preimages in $\SL_2(\C)$ of the finite subgroups $A_4$ (tetrahedral group), $S_4$ (octahedral group), and $A_5$ (icosahedral group) of $\PGL_2(\C)$, respectively.
\end{enumerate}

We refer to the Appendix (after the statement of Theorem A.3) for an explicit description of the finite groups 
$B A_4$, $B S_4$, and $B A_5$.

\begin{definition}
A representation $\pi\,\to\,\SL_2(\C)$ of a group $\pi$ is:
	\begin{enumerate}
\item[\textup{(1)}] \emph{Zariski dense} if its image is not contained in a proper algebraic subgroup of $\SL_2$ defined over $\C$,

\item[\textup{(2)}] \emph{diagonal} if it factors through the inclusion $i\,:\,\C^\times\,\to\,\SL_2(\C)$ of the maximal 
torus consisting of diagonal matrices,

\item[\textup{(3)}] \emph{dihedral} if it factors through the inclusion $j\,:\,D_\infty\,\to\,\SL_2(\C)$, 

\item[\textup{(4)}] \emph{finite} if its image is finite,

\item[\textup{(5)}] \emph{unitarizable} if its image is conjugate to a subgroup of  $\SU(2)$,

\item[\textup{(6)}] \emph{reducible} if its image preserves a subspace of $\C^2$, 

\item[\textup{(7)}] \emph{irreducible} if it is not reducible,

\item[\textup{(8)}] \emph{elementary} if it is unitary or reducible or dihedral, 

\item[\textup{(9)}] \emph{non-elementary} if it is not elementary, and

\item[\textup{(10)}] \emph{semisimple} if $\C^2$ is a direct sum of $\pi$-modules under the representation (or, 
equivalently, the representation is irreducible or conjugate to a diagonal representation.)
\end{enumerate}
\end{definition}

\subsection{Character varieties} \label{sect:2.2}

Given a finitely presented group $\pi$, let us define the \emph{$\SL_2$-representation variety} $\Rep(\pi)$ as
the complex affine scheme determined by the functor
$$A\,\mapsto\,\Hom(\pi,\,\SL_2(A))$$
for every commutative $\C$-algebra $A$. Given a sequence of generators of $\pi$ with $m$ elements, we have a presentation of $\Rep(\pi)$ as a closed subscheme of $\SL_2^m$ defined by equations coming from relations among the generators. For each $a\in\pi$, let $\tr_a$ be the regular function on $\Rep(\pi)$ given by $\rho\mapsto\tr\rho(a)$. The \emph{character variety} of $\pi$ over $\C$ is the affine invariant theoretic quotient
$$X(\pi)\,=\,\Rep(\pi)\git\SL_2\,=\,\Spec\C[\Rep(\pi)]^{\SL_2(\C)}$$
under the  conjugation action of $\SL_2$. The complex points of $X(\pi)$ parametrize the isomorphism
classes of semisimple representations $\pi\to \SL_2(\C)$, or equivalently the Jordan equivalence classes
of representations $\pi\to\SL_2(\C)$ (see e.g.~\cite[Proposition 6.1]{simpson}). For each $a\,\in\,\pi$ the
regular function $\tr_{a}$ evidently descends to a regular function on $X(\pi)$. The
scheme $X(\pi)$ has a natural model over $\Z$. We refer to \cite{horowitz}, \cite{ps}, \cite{saito} for details.

\begin{example}
\label{exfree}
We refer to Goldman \cite{goldman2} for details of the examples below. Let $F_m$ denote the free group on $m\geq1$ generators $a_1,\cdots,a_m$.
\begin{enumerate}
	\item[(1)] We have $\tr_{a_1}\,:\,X(F_1)\,\simeq\,\A^1$.
	\item[(2)] We have $(\tr_{a_1},\,\tr_{a_2},\,\tr_{a_1a_2})\,:\,X(F_2)\,\simeq\,\A^3$ by Fricke
\cite[Section 2.2]{goldman2}.
	\item[(3)] The coordinate ring $\Z[X(F_3)]$ is the quotient of the polynomial ring
	$$\Z[\tr_{a_1},\tr_{a_2},\tr_{a_3},\tr_{a_1a_2},\tr_{a_2a_3},\tr_{a_1a_3},\tr_{a_1a_2a_3},\tr_{a_1a_3a_2}]$$
	by the ideal generated by two elements
$$
\tr_{a_1a_2a_3}+\tr_{a_1a_3a_2}-(\tr_{a_1a_2}\tr_{a_3}+\tr_{a_1a_3}\tr_{a_2}+\tr_{a_2a_3}\tr_{a_1}-\tr_{a_1}\tr_{a_2}\tr_{a_3})
$$
and
\begin{align*}
\tr_{a_1a_2a_3}\tr_{a_1a_3a_2}&-\{(\tr_{a_1}^2+\tr_{a_2}^2+\tr_{a_3}^2)+(\tr_{a_1a_2}^2+\tr_{a_2a_3}^2+\tr_{a_1a_3}^2)&\\
&\quad -(\tr_{a_1}\tr_{a_2}\tr_{a_1a_2}+\tr_{a_2}\tr_{a_3}\tr_{a_2a_3}+\tr_{a_1}\tr_{a_3}\tr_{a_1a_3})\\
&\quad +\tr_{a_1a_2}\tr_{a_2a_3}\tr_{a_1a_3}-4\}.
\end{align*}
In particular, $\tr_{a_1a_2a_3}$ and  $\tr_{a_1a_3a_2}$ are integral over the polynomial subring $\Z[\tr_{a_1},\tr_{a_2},\tr_{a_3},\tr_{a_1a_2},\tr_{a_2a_3},\tr_{a_1a_3}]$.
\end{enumerate}
\end{example}

We record the following, which is attributed by Goldman, \cite{goldman2}, to Vogt \cite{vogt}.

\begin{lemma}
\label{rellem}
Given a finitely generated group $\pi$ and $a_1,\,a_2,\,a_3,\,a_4\,\in\, \pi$, the following holds:
\begin{align*}
2{\tr_{a_1a_2a_3a_4}}&={\tr_{a_1}}{\tr_{a_2}}{\tr_{a_3}}{\tr_{a_4}}+{\tr_{a_1}}{\tr_{a_2a_3a_4}}+{\tr_{a_2}}{\tr_{a_3a_4a_1}}+{\tr_{a_3}}{\tr_{a_4a_1a_2}}\\
&\quad +{\tr_{a_4}}{\tr_{a_1a_2a_3}}+{\tr_{a_1a_2}}{\tr_{a_3a_4}}+{\tr_{a_4a_1}}{\tr_{a_2a_3}}-{\tr_{a_1a_3}}{\tr_{a_2a_4}}\\
&\quad -{\tr_{a_1}}{\tr_{a_2}}{\tr_{a_3a_4}}-{\tr_{a_3}}{\tr_{a_4}}{\tr_{a_1a_2}}-{\tr_{a_4}}{\tr_{a_1}}{\tr_{a_2a_3}}-{\tr_{a_2}}{\tr_{a_3}}{\tr_{a_4a_1}}.
\end{align*}
\end{lemma}

The above computations imply the following fact.

\begin{fact}
\label{fact}
If $\pi$ is a group generated by $a_1,\,\cdots,\,a_m$, then $\Q[X(\pi)]$ is generated as a $\Q$-algebra by the collection $\{\tr_{a_{i_1}\cdots a_{i_k}}\,\mid\,1\leq i_1<\cdots<i_k\leq m\}_{1\leq k\leq 3}$.
\end{fact}

\begin{remark}
Using Fact \ref{fact}, we will show (Lemma \ref{simplefinitelem}) that an $\SL_2(\C)$-representation $\rho$ of a surface group has finite mapping class group orbit in the $\SL_2$-character variety if and only if the set of its traces along simple closed curves is finite.
\end{remark}

The construction of $X(\pi)$ is functorial with respect to the group
$\pi$. Given a homomorphism $f\,:\,\pi\,\to\,\pi'$ of finitely presented groups, the corresponding
morphism $f^*\,:\,X(\pi')\,\to\, X(\pi)$ sends a representation $\rho$ to the semisimplification of
$\rho\circ f$. In particular, the automorphism group $\Aut(\pi)$ of $\pi$ naturally acts
on $X(\pi)$. This action naturally factors through the outer automorphism group $\Out(\pi)$,
owing to the fact that $\tr_{aba^{-1}}\,=\,\tr_{b}$ for every $a,b\in\pi$.

Let $i\,:\,\C^\times\,\to\,\SL_2(\C)$ and $j\,:\,D_\infty\,\to\,\SL_2(\C)$ be the inclusion maps of the diagonal
maximal torus and the infinite dihedral group, respectively. They induce $\Out(\pi)$-equivariant maps
\begin{align*}
\tag{$*$}
i_*\,:\,\Hom(\pi,\C^\times)\,\to\, X(\pi)(\C)\quad\text{ and }\quad j_*\,:\,\Hom(\pi,D_\infty)/D_\infty\,\to\, X(\pi)(\C).
\end{align*}

\begin{lemma}
\label{ffib}
The following two hold.
\begin{enumerate}
\item The map $i_*\,:\,\Hom(\pi,\C^\times)\,\to\, X(\pi)(\C)$ in $(*)$ has finite fibers.
\item The map $j_*\,:\,\Hom(\pi,D_\infty)/D_\infty\,\to\, X(\pi)(\C)$ in $(*)$ has finite fibers.
\end{enumerate}
\end{lemma}

\begin{proof}
(1) Let $\rho_1,\,\rho_2\,:\,\pi\,\to\,\C^\times$ be characters such that $gi_*(\rho_1)g^{-1}\,=\,i_*(\rho_2)$ for
some $g\,\in\,\SL_2(\C)$. Without  loss of generality, we may assume
that $\rho_1(x)\,\neq\,\pm1$ for some $x\in\pi$, since otherwise the image of $\rho_1$
is finite and we are done. Writing $g\,=\,\left[\begin{smallmatrix}a & b\\ c &d\end{smallmatrix}\right]$ and $\rho_1(x)=\lambda$ with $\lambda\in\C^\times\setminus\{\pm1\}$, we have
$$\begin{bmatrix}a &b\\ c &d\end{bmatrix}\begin{bmatrix}\lambda & 0\\0 & \lambda^{-1}\end{bmatrix}\begin{bmatrix}d &-b\\ -c &a\end{bmatrix}=\begin{bmatrix}\lambda ad-\lambda^{-1}bc & (\lambda^{-1}-\lambda)ab\\(\lambda-\lambda^{-1})cd&\lambda^{-1}ad-\lambda bc\end{bmatrix}.$$
For the matrix on the right hand side to be diagonal, we must thus have $a\,=\,d\,=\,0$ or $b\,=\,c\,=\,0$
since $\lambda\,\neq\,\pm1$. If $a\,=\,d\,=\,0$, then $\rho_2\,=\,\rho_1^{-1}$. If $b\,=\,c\,=\,0$, then
$\rho_1\,=\,\rho_2$. This proves (1).

(2) Let $\rho_1,\,\rho_2\,:\,\pi\,\to\, D_\infty$ be representations such that 
$gj_*(\rho_1)g^{-1}\,=\,j_*(\rho_2)$ for some $g\,\in\,\SL_2(\C)$. Consider first the case where
$\tr\rho_1(x)\,\in\,\{0,\pm2\}$ for every $x\,\in\,\pi$. In this case, up to conjugation in $D_\infty$ the
image of $\rho_1$ must lie in the finite group
$$G\,=\,\left\langle\begin{bmatrix}
0 & i^k\\ -i^{-k} & 0
\end{bmatrix},\begin{bmatrix}
i^k & 0\\ 0 & {i^{-k}}
\end{bmatrix}:k\in\Z\right\rangle.$$
To see this, note first that clearly the diagonal elements in the image of $\rho_1$ belong to $G$. Next, suppose
we have (and fix) some $x\,\in\,\pi$ such that $\rho_1(x)$ is not diagonal. Up to conjugating $\rho_1$ by an
element of $D_\infty$ we may assume
$$\rho_1(x)\,=\,\begin{bmatrix}0 & 1\\ -1 & 0\end{bmatrix}.$$
It follows that any other $y\,\in\,\pi$ with non-diagonal $\rho_1(y)$ must be in $G$, since otherwise
$\tr\rho(xy)\,\notin\,\{0,\,\pm2\}$. Thus, we conclude that if $\tr\rho_1(x)\,\in\,\{0,\,\pm2\}$ for every
$x\,\in\,\pi$ (and the same holds for $\rho_2$ since $\rho_1$ and $\rho_2$ are assumed to be
$\SL_2(\C)$-conjugate), then $\rho_1$ and $\rho_2$ are each $D_\infty$-conjugate to a representation
with image in the finite group $G$. Since $\pi$ is finitely generated, there are
only finitely many homomorphisms $\pi\to G$, and from this our claim follows.

It remains to treat the case where $\tr\rho_1(x)\,\notin\,\{0,\,\pm2\}$ for some $x\,\in\,\pi$. Now note that
$\rho_1(x)$ must be diagonal. The equation 
$g\rho_1(x)g^{-1}\,=\,\rho_2(x)$ shows that, by the same computation as in the proof of part (1), we have $g\,\in\, 
D_\infty$. This proves (2).
\end{proof}

\subsection{Surfaces} \label{sect:2.3}

Here, we set our notational convention and terminology for various topological notions. Throughout this paper, 
a \emph{surface} is the complement of a finite collection of interior points in a compact oriented topological 
manifold of dimension 2, with or without boundary.

A \emph{simple closed curve} on a surface is an embedded copy of an unoriented circle. We shall often refer to 
simple closed curves simply as curves (since immersed curves will not be important in this paper). Given a 
surface $\Sigma$, a curve in $\Sigma$ is \emph{nondegenerate} if it does not bound a disk on $\Sigma$. A curve 
in the interior of $\Sigma$ is \emph{essential} if it is nondegenerate, does not bound a once-punctured disk on 
$\Sigma$, and is not isotopic to a boundary curve of $\Sigma$. Given a surface $\Sigma$ and an essential curve 
$a\,\subset\,\Sigma$, we denote by $\Sigma|a$ the surface obtained by cutting $\Sigma$ along $a$. An essential 
curve $a\,\subset\,\Sigma$ is \emph{separating} if the two boundary curves of $\Sigma|a$ corresponding to $a$ are 
in different connected components, and \emph{nonseparating} otherwise.

Let $\Sigma$ be a surface of genus $g$ with $n$ punctures or boundary curves. We shall denote by $\Mod(\Sigma)$ 
the (pure) mapping class group of $\Sigma$. By definition, it is the group of isotopy classes of orientation 
preserving homeomorphisms of $\Sigma$ fixing the punctures and boundary points individually. Given a simple 
closed curve $a\subset\Sigma$, we shall denote by $\tw_a\,\in\,\Mod(\Sigma)$ the associated (left) Dehn twist.

 We  define the character variety of $\Sigma$ (cf.\ Section \ref{sect:2.2}) to be
$$X(\Sigma)\,=\,X(\pi_1(\Sigma)).$$
The complex points of $X(\Sigma)$ can be seen as parametrizing the isomorphism classes of semisimple 
$\SL_2(\C)$-local systems on $\Sigma$. Note that a simple closed curve $a\,\subset\, \Sigma$ unambiguously defines 
a function $\tr_a$ on $X(\Sigma)$, coinciding with $\tr_\alpha$ for any loop $\alpha\,\in\,\pi_1(\Sigma)$ freely 
homotopic to a parametrization of $a$.

Given a continuous map $f\,:\,\Sigma'\,\to\,\Sigma$ of surfaces, we have an induced morphism of character varieties 
$f^*\,:\,X(\Sigma)\,\to\, X(\Sigma')$ depending only on the homotopy class of $f$. In particular, the mapping class 
group $\Mod(\Sigma)$ acts naturally on $X(\Sigma)$ by precomposition: given $\rho\,:\,\pi_1(\Sigma)\,\to\,\SL_2(\C)$ 
and $\gamma\,\in\, \Mod(\Sigma)$, the class of $\rho$ in $X(\C)$ is mapped to the class of $\rho\circ\gamma_*$ 
where $\gamma_*$ is the outer automorphism of $\pi_1(\Sigma)$ determined by $\gamma$ with any choice of 
basepoint. If $\Sigma'\subset\Sigma$ is a subsurface, the induced morphism on character varieties is 
$\Mod(\Sigma')$-equivariant for the induced morphism $\Mod(\Sigma')\,\to\,\Mod(\Sigma)$ of mapping class groups. In 
particular, if a semisimple $\SL_2(\C)$-representation of $\pi_1(\Sigma)$ has a finite $\Mod(\Sigma)$-orbit in 
$X(\Sigma)$, then its restriction to any subsurface $\Sigma'\subset\Sigma$ has a finite $\Mod(\Sigma')$-orbit 
in $X(\Sigma')$.

\subsection{Loop configurations} \label{sect:2.4}

Let $\Sigma$ be a surface of genus $g$ with $n$ punctures. We fix a base point in $\Sigma$. For convenience, we 
shall say that a sequence $\ell\,=\,(\ell_1,\,\cdots,\,\ell_m)$ of based loops on $\Sigma$ is \emph{clean} if each loop 
is simple and the loops pairwise intersect only at the base point.

\begin{example}
\label{exgen}
Recall the standard presentation of the fundamental group
$$\pi_1(\Sigma)\,=\,\langle a_1,d_1,\cdots,a_g,d_g,c_1,\cdots,c_n\,
\mid\,[a_1,d_1]\cdots[a_g,d_g]c_1\cdots c_n\rangle.$$
We can choose (the based loops representing) the generators so that the sequence of loops
$(a_1,d_1,\cdots,a_g,d_g,c_1,\cdots,c_n)$ is clean. For $i\,=\,1,\,\cdots,\,g$, let $b_i$ be the based simple
loop parametrizing the curve underlying $d_i$ with the opposite orientation. Note that
$(a_1,\,b_1,\,\cdots,\,a_g,\,b_g,\,c_1,\,\cdots,\,c_n)$ is a clean sequence with the property that any product of
distinct elements preserving the cyclic ordering on the sequence, such as $a_1b_g$ or $a_1a_2b_2b_g$ or
$b_gc_na_1$, can be represented by a simple loop in $\Sigma$. We shall refer to
$(a_1,\,b_1,\,\cdots,\,a_g,\,b_g,\,c_1,\,\cdots,\,c_n)$ as \emph{an optimal sequence of generators} of
$\pi_1(\Sigma)$. See Figure \ref{fig1} for an illustration of the optimal generators for $(g,\,n)\,=\,(2,\,1)$.
\end{example}

\begin{lemma}\label{simplefinitelem}
Let $\rho\,:\,\pi_1(\Sigma)\,\to\,\SL_2(\C)$ be a representation. Then $\rho$ has finite $\Mod(\Sigma)$-orbit
$X(\Sigma)$ if and only if
$$\Scal\,=\,\{\tr_a(\rho):\text{$a\subset\Sigma$ simple closed curve}\}$$
is a finite set.
\end{lemma}

\begin{proof}
Let $(a_1,\,\dots,\,a_{2g+n})$ be an optimal sequence of generators for $\pi_1(\Sigma)$. By Fact \ref{fact}, the
coordinate ring $\Q[X(\Sigma)]$ of $X(\Sigma)$ is generated over $\Q$ by the collection of trace functions
$$\{\tr_{a_{i_1}\cdots a_{i_k}}\,\mid\,1\leq i_1<\cdots<i_k\leq 2g+n\}_{1\leq k\leq 2g+n},$$
giving us a closed immersion
$$j\,=\,(\tr_{a_{i_1}\cdots a_{i_k}})\,:\,X(\Sigma)\,\hookrightarrow\,\A^{2^{2g+n}-1}.$$
By our observation about optimal sequences of generators in Example \ref{exgen}, every $a_{i_1}\cdots a_{i_k}$
appearing above is homotopic to a simple loop. Since $\tr_a(\gamma^*\rho)\,=\,\tr_{\gamma(a)}(\rho)$ for
every simple closed curve $a\,\subset\,\Sigma$ and $\gamma\,\in\,\Mod(\Sigma)$, we see that if $\Scal$ is
finite then the coordinates (under the immersion $j$) of the points in the $\Mod(\Sigma)$-orbit of $\rho$
belong to a finite set, and hence $\rho$ has finite $\Mod(\Sigma)$-orbit in $X(\Sigma)$. Conversely, if
$\rho$ has finite $\Mod(\Sigma)$-orbit in $X(\Sigma)$, $\{b_1,\dots,b_r\}$ is a collection of representatives of mapping class group equivalence classes of simple closed curves on $\Sigma$ (considered up to isotopy), and $\Scal'$ is the finite set of coordinates
of the images of the points in the orbit under the morphism $(\tr_{b_i}):X(\Sigma)\to\A^r$, then we have $\Scal\,=\,\Scal'$ which is therefore finite.
\end{proof}

\begin{figure}[ht]
    \centering
    \includegraphics{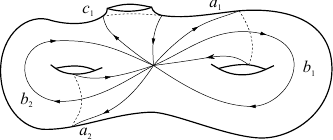}
    \caption{Optimal generators for $(g,n)=(2,1)$}
    \label{fig1}
\end{figure}

\begin{definition}
A \emph{loop configuration} is a planar graph consisting of a single vertex $v$ and a finite cyclically ordered sequence of directed rays, equipped with a bijection between the set of rays departing from $v$ and the set of rays arriving at $v$. We denote by $L_{g,n}$ the loop configuration whose sequence of rays is of the form
$$(a_1,b_1,\overline{a}_1,\overline{b}_1,\cdots,a_g,b_g,\overline{a}_g,\overline{b}_g,c_1,
\overline{c}_1,\cdots,c_n,\overline{c}_n),$$
where $a_i,b_i,c_i$ are the rays directed away from $v$, respectively corresponding to the rays
$\overline{a}_i,\overline{b}_i, \overline{c}_i$ directed towards $v$. An isomorphism of loop configurations is an isomorphism of planar graphs respecting the bijections between the departing and arriving rays. See Figure \ref{fig4} for an illustration of $L_{2,1}$.
\end{definition}

Given a clean sequence $\ell=(\ell_1,\cdots,\ell_m)$ of loops on $\Sigma$, we have an associated loop configuration $L(\ell)$, obtained by taking a sufficiently small open neighborhood of the base point and setting the departing and arriving ends of the loops $\ell_i$ to correspond to each other. For example, if $(a_1,b_1,\cdots,a_g,b_g,c_1,\cdots,c_n)$ is a sequence of optimal generators for $\pi_1(\Sigma)$, then
$$L(a_1,b_1,\cdots,a_g,b_g,c_1\cdots,c_n)\simeq L_{g,n}.$$

\begin{figure}[ht]
    \centering
    \includegraphics{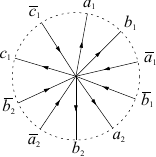}
    \caption{Loop configuration $L_{2,1}$}
    \label{fig4}
\end{figure}

\begin{definition}
Let $h$ and $m$ be nonnegative integers. A sequence of based loops $\ell\,=\,(\ell_1,\,\cdots,\,\ell_{2h+m})$ on $\Sigma$ is said to be \emph{in $(h,m+1)$-position} if it is homotopic term-wise to a clean sequence $\ell'=(\ell_1',\cdots,\ell_{2h+m}')$ such that $L(\ell')\simeq L_{h,m}$. We denote by $\Sigma(\ell)\subset \Sigma$ the (isotopy class of a) subsurface of genus $h$ with $m+1$ boundary curves obtained by taking a small closed tubular neighborhood of the union of the simple curves underlying $\ell_1',\cdots,\ell_{2h+m}'$ in $\Sigma$.
\end{definition}

\begin{remark}
It is worth emphasizing that the terminology above is intended to evoke the topological type of the surface $\Sigma(\ell)$, and not the isomorphism type of the loop configuration.
\end{remark}

\section{Non-Zariski-dense representations} \label{sect:3}
Let $\Sigma$ be a surface of genus $g\geq1$ with $n\geq0$ punctures. The purpose of this section is to characterize representations $\pi_1(\Sigma)$ with non-Zariski-dense image in $\SL_2(\C)$ that have finite mapping class group orbit in the character variety $X(\Sigma)$. An irreducible representation $\rho:\pi_1(\Sigma)\to \SL_2(\C)$ will be called \emph{special dihedral} if it factors through $D_\infty$ and there is a nonseparating essential curve $a$ in $\Sigma$ such that the restriction $\rho|(\Sigma\setminus a)$ is diagonal.
The main result of this section is the following.

\begin{proposition}\label{g1prop}
Let $\Sigma$ be a surface of genus $g\geq1$ with $n\geq0$ punctures. Let $\rho\,:\,\pi_1(\Sigma)\,\to\,\SL_2(\C)$
be a semisimple representation whose image is not Zariski-dense in $\SL_2(\C)$. Then $\rho$ has a finite
mapping class group orbit in $X(\Sigma)$ if and only if one of the following
holds:
\begin{enumerate}
\item[\textup{(1)}] $\rho$ is a finite representation.
\item[\textup{(2)}] $g\,=\,1$ and $\rho$ is special dihedral up to conjugation.
\end{enumerate}
\end{proposition}

It is evident that, if the image of $\rho\,:\,\pi_1(\Sigma)\,\to\,\SL_2(\C)$ belongs to one of the finite 
groups $BA_4$, $BS_4$, and $BA_5$, then its orbit in $X(\Sigma)$ is finite. From the discussion in Section 
\ref{sect:2.1} and Lemma \ref{ffib}, to prove Proposition \ref{g1prop} it remains to understand the finite 
mapping class group orbits on $\Hom(\pi_1(\Sigma),\,\C^\times)$ and $\Hom(\pi_1(\Sigma),\,D_\infty)/D_\infty$. 
Proposition \ref{g1prop} thus follows by combining Lemmas \ref{redlemm}, \ref{dilem2}, and \ref{dilem} below.\\

\noindent{\bf Case 1: Diagonal Representations.}

\begin{lemma}\label{redlemm}
Assume $\Sigma$ is a surface of genus $g\geq1$ with $n\geq0$ punctures. A representation $\rho\,:\,
\pi_1(\Sigma)\,\to\,\C^\times$ has finite (respectively, ~bounded) mapping class group orbit in $\Hom(\pi_1(\Sigma),
\,\C^\times)$ if and only if it has finite (respectively, bounded) image.
\end{lemma}

\begin{proof}
Let $\rho\,:\,\pi_1(\Sigma)\,\to\,\C^\times$ be a representation with finite (respectively, bounded) mapping class
group orbit in
$\Hom(\pi_1(\Sigma),\,\C^\times)$. Let $(a_1,\,b_1,\,\cdots,\,a_g,\,b_g,\,c_1,\,\cdots,\,c_n)$ be the optimal
generators of $\pi_1(\Sigma)$. Assume first that $n\,=\,0$ or $1$. By considering the effect of Dehn twist
along the curve underlying $a_1$ on the curve underlying $b_1$, we conclude that $\rho(a_1)$ must be
torsion (respectively, have absolute value $1$). Applying the same argument to the other loops in the sequence
of optimal generators, we conclude that $\rho$ is finite (respectively, bounded) if $n\,\leq\,1$, as desired.

Thus, only the case $n\geq2$ remains. Since $(a_1,\,b_1,\,\cdots,\,a_g,\,b_g)$ is in $(g,\,1)$-position,
the above analysis shows that $\rho(a_i)$ and $\rho(b_i)$ are roots of unity (respectively, have absolute
value 1) for $i\,=\,1,\,\cdots,\,g$. Similarly, we see that the sequence
$$L_i\,=\,(a_1,\,b_1,\,\cdots,\,a_{g-1},\,b_{g-1},\,a_g,\,b_gc_i)$$
is in $(g,1)$-position for every $i\,=\,1,\,\cdots,\,n$. Since $\rho$ restricted to the surface $\Sigma(L_i)$ must
have finite (respectively, bounded) $\Mod(\Sigma(L_i))$-orbit,
it follows that $\rho(c_i)$ is a root of unity (respectively, has absolute value $1$) for $i\,=\,1,\,\cdots,\,n$
as well. This shows that $\rho$ has finite (respectively, bounded) image, as desired.
\end{proof}

\noindent{\bf Case 2: Dihedral Representations.}
We first prove Proposition \ref{g1prop} for surfaces of genus $g\geq2$ in Lemma \ref{dilem2} below.

\begin{lemma}
\label{dilem2}
Assume $\Sigma$ is a surface of genus $g\,\geq\,2$ with $n\,\geq\,0$ punctures. A representation $\rho\,:\,
\pi_1(\Sigma)\,\to\, D_\infty$ has finite mapping class group orbit in $\Hom(\pi_1(\Sigma),\,D_\infty)/D_\infty$
if and only if it has finite image.
\end{lemma}

\begin{proof}
The ``\emph{if}'' direction is clear, and we now prove the converse. Let $\rho\,:\,\pi_1(\Sigma)\,\to\, D_\infty$
be a representation with finite mapping class group orbit in $\Hom(\pi_1(\Sigma),\,D_\infty)/D_\infty$. We have
a short exact sequence $1\to\C^\times\,\to\, D_\infty\,\to\, \Z/2\Z\to0$, where the
homomorphism $D_\infty\,\to\, \Z/2\Z$ is given by
$$\begin{bmatrix}c & 0\\ 0 & c^{-1}\end{bmatrix}\mapsto 0\quad\text{and}\quad \begin{bmatrix}0 & c\\ -c^{-1} & 0\end{bmatrix}\mapsto 1\quad\text{for all $c\in\C^\times$.}$$
This gives us a $\Mod(\Sigma)$-equivariant map
$$\Hom(\pi_1(\Sigma),\,D_\infty)/D_\infty\,\to\,\Hom(\pi_1(\Sigma),\,\Z/2\Z).$$ The fiber of this map above the
zero homomorphism consists of those points given by diagonal representations, to which Lemma \ref{redlemm}
applies, noting that the map $\Hom(\pi_1(\Sigma),\,\C^\times)\,\to\,\Hom(\pi_1(\Sigma),\,D_\infty)/D_\infty$ has finite
fibers. It suffices to consider the case where $\rho\,:\,\pi_1(\Sigma)\,\to\, D_\infty$ is not in the fiber over
the zero homomorphism. This assumption implies that there is a nonseparating loop $a_1$ on $\Sigma$ so
that $\rho(a_1)$ is not diagonal. Let us extend such $a_1$ to an optimal sequence of generators
$$(a_1,\,b_1,\,\cdots,\,a_g,\,b_g,\,c_1,\,\cdots,\,c_n)$$
for $\pi_1(\Sigma)$ (see Section 2.4 for the definition). Note that $L\,=\,(a_1,\,b_1,\,\cdots,\,a_g,\,b_g)$ is in
$(g,\,1)$-position, and we have a $\Mod(\Sigma(L))$-equivariant homomorphism
$$\Hom(\pi_1(\Sigma),\,\Z/2\Z)\,\to\,\Hom(\pi_1(\Sigma(L)),\,\Z/2\Z).$$
The action of $\Mod(\Sigma(L))$ on $\Hom(\pi_1(\Sigma(L)),\,\Z/2\Z)$ factors through the projection
$$\Mod(\Sigma(L))\twoheadrightarrow\Sp(2g,\Z/2\Z).$$ From the transitivity of $\Sp(2g,\Z/2\Z)$ on $(\Z/2\Z)^{2g}$ away from the origin, it follows that, up to $\Mod(\Sigma)$-action, we may assume that
$$\overline{\rho}(a_1)=1\quad\text{and}\quad\overline{\rho}(b_1)=\overline{\rho}(a_2)=
\cdots=\overline{\rho}(b_g)=0\, ,$$
where $\overline{\rho}$ is the image of $\rho$ in $\Hom(\pi_1(\Sigma),\,\Z/2\Z)$. Up to conjugation by an
element in $D_\infty$, we may moreover assume that
$$\rho(a_1)=\begin{bmatrix}0 & 1\\ -1 & 0\end{bmatrix}.$$
It suffices to show that the entries of the matrices $\rho(b_1)$, $\rho(a_2)$, $\rho(b_2)$, $\cdots$, $\rho(a_g)$, $\rho(b_g)$, $\rho(c_1)$, $\cdots$, $\rho(c_n)$ are roots of unity. Let $i_1<\cdots<i_q$ be precisely the indices in $\{1,\cdots,n\}$ such that $\overline{\rho}(c_{i_j})=0$. Since
$$L'=(a_2,b_2,\cdots,a_g,b_g,c_{i_1},\cdots,c_{i_q},b_1)$$
is in $(g-1,q+2)$-position, and the restriction of $\rho$ to $\Sigma(L')$ is diagonal, we see that $\rho(b_1)$, $\rho(a_2)$, $\rho(b_2)$, $\cdots$, $\rho(a_g)$, $\rho(b_g)$, $\cdots$, $\rho(c_{i_1})$, $\cdots$, $\rho(c_{i_q})$ are torsion by Lemma \ref{redlemm}. Let us now take $i\in\{1,\cdots,n\}\setminus\{i_1,\cdots,i_q\}$. The restriction of $\rho$ to the subsurface $\Sigma(c_ia_1,b_1)$ of genus $1$ with $1$ boundary curve is diagonal. Writing
$$\rho(c_i)=\begin{bmatrix}0 & \lambda_i\\-\lambda_i^{-1}&0\end{bmatrix}$$
with $\lambda_i\in\C^\times$, we have
$$\rho(c_ia_1)=\begin{bmatrix}0 & \lambda_i\\ -\lambda_i^{-1} & 0\end{bmatrix}\begin{bmatrix}0 & 1\\ -1 & 0\end{bmatrix}=\begin{bmatrix}-\lambda_i&0\\ 0 & -\lambda_i^{-1}\end{bmatrix}.$$
It follows from Lemma \ref{redlemm} that $\lambda_i$ is a root of unity. This completes the proof that the entries of $\rho(b_1)$, $\rho(a_2)$, $\rho(b_2)$, $\cdots$, $\rho(a_g)$, $\rho(b_g)$, $\rho(c_1)$, $\cdots$, $\rho(c_n)$ are roots of unity, and hence the image of $\rho$ is finite, as desired.
\end{proof}

Before proving the case $g\,=\,1$ of Proposition \ref{g1prop} in Lemma \ref{dilem} below, we record the following 
elementary lemma.

\begin{lemma}
\label{intlem}
Let $\Sigma$ be a surface of genus $1$ with $n\geq0$ punctures. Suppose that $\rho\,:\,\pi_1(\Sigma)
\,\to\, D_\infty$ is special dihedral. Given a pair of loops $(a,\,b)$ in $(1,\,1)$-position on $\Sigma$, at
least one of $\rho(a)$ and $\rho(b)$ is not diagonal.
\end{lemma}

\begin{proof}
Let $\rho\,:\,\pi_1(\Sigma)\,\to\, D_\infty$ be special dihedral. We argue by contradiction. Assume that 
$(a_1,\,b_1)$ is a pair of loops in $(1,\,1)$-position on $\Sigma$ with both $\rho(a_1)$ and 
$\rho(b_1)$ diagonal. We can complete $(a_1,\,b_1)$ to a sequence of optimal generators 
$(a_1,\,b_1,\,c_1,\,\cdots,c_n)$ of $\pi_1(\Sigma)$. Since $\rho$ is special dihedral, the matrices 
$\rho(c_1),\,\cdots,\,\rho(c_n)$ must be diagonal (noting that the property of a matrix  being diagonal is not 
changed under conjugation by an element in $D_\infty$). This implies that $\rho$ is in fact 
diagonal, contradicting the hypothesis that $\rho$ is irreducible.
\end{proof}

\begin{lemma}
\label{dilem}
Let $\Sigma$ be a surface of genus $1$ with $n\,\geq\,0$ punctures. A representation $\rho\,:\,\pi_1(\Sigma)\,\to
\,D_\infty$ has finite mapping class group orbit in $\Hom(\pi_1(\Sigma),\,D_\infty)/D_\infty$ if and only
if it is finite or special dihedral.
\end{lemma}

\begin{proof}
The same argument as in Lemma \ref{dilem2} shows that if $\rho\,:\,\pi_1(\Sigma)\,\to\, D_\infty$ has
finite $\Mod(\Sigma)$-orbit in $\Hom(\pi_1(\Sigma),\,D_\infty)/D_\infty$, then $\rho$ is finite or special
dihedral. Indeed, by Lemma \ref{redlemm} we are done if $\rho$ is diagonal, so let us assume otherwise. This
implies there is a nonseparating loop $a_1$ such that $\rho(a_1)$ is not diagonal. We extend $a_1$ to an optimal
sequence
$$(a_1,\,b_1,\,c_1,\,\dots,\,c_n)$$
of generators for $\pi_1(\Sigma)$. Without loss of generality, we may assume $\rho(b_1)$ is diagonal (otherwise, replace $b_1$ by a simple loop homotopic to $b_1a_1$ and complete $(a_1,b_1a_1)$ to an optimal sequence of generators for $\pi_1(\Sigma)$). Up to conjugation in $D_\infty$, we can also assume
$$\rho(a_1)=\begin{bmatrix}0 & -1\\ 1 & 0\end{bmatrix}.$$
If $\rho(c_1),\dots,\rho(c_n)$ are all diagonal, then $\rho$ is special dihedral as desired. So assume otherwise; we need to show that $\rho$ is finite. Suppose $\rho(c_i)$ is not diagonal, and say
$$\rho(c_i)=\begin{bmatrix}0 & \lambda_i\\ -\lambda_i^{-1} & 0\end{bmatrix}
$$
for some $\lambda_i\,\in\,\C^\times$. Note that $(c_ia_1,\,b_1)$ is in $(1,\,1)$-position and the restriction of $\rho$ to
$\Sigma(c_ia_1,\,b_1)$ is diagonal, so $\lambda_i$ as well as the eigenvalues of $b_i$ are roots of unity. This holds true for all $c_i$ with non-diagonal $\rho(c_i)$. In what follows, fix one such $c_i$ with $\rho(c_i)$ non-diagonal. For any $c_j$ with $j<i$ such that $\rho(c_j)$ is diagonal, we see that $(c_ia_1,b_1c_j)$ is in $(1,1)$-position and the restriction of $\rho$ to $\Sigma(c_ia_1,b_1c_j)$ is diagonal, so the eigenvalues of $\rho(c_j)$ are roots of unity. Similarly, for any $c_j$ with $i<j$ such that $\rho(c_j)$ is diagonal, we see that $(c_ic_ja_1,b_1)$ is in $(1,1)$-position and the restriction of $\rho$ to $\Sigma(c_ic_ja_1,b_1)$ is diagonal, so the eigenvalues of $\rho(c_j)$ are roots of unity. Combining these observations, it follows that $\rho$ has finite image, as
desired. Thus, we have prove the ``only if'' direction of the lemma.

Let now $\rho\,:\,\pi_1(\Sigma)\,\to\, D_\infty$ be a special dihedral representation. We shall show that 
$\rho$ has finite mapping class group orbit in $\Hom(\pi_1(\Sigma),\,D_\infty)/D_\infty$, or equivalently in the 
character variety $X(\Sigma)$. By Lemma \ref{simplefinitelem}, it suffices to show that the set 
$$\{\tr\rho(a)\,\mid\,a\,\subset\,\Sigma\text{ essential curve}\}\subseteq\C$$ is finite.

Let $(a_1,\,b_1,\,c_1,\,\cdots,\,c_n)$ be optimal generators of $\pi_1(\Sigma)$. Since $\rho$ is special 
dihedral, $\rho(c_1)$, $\cdots$, $\rho(c_n)$ are diagonal matrices. Suppose $a$ is a separating essential curve 
in $\Sigma$. Since one of the connected components of $\Sigma|a$ is a surface of genus zero (with one boundary 
curve corresponding to $a$ and finitely many punctures), the monodromy of $\rho$ along $a$ is a product of 
conjugates of a subset of $\{\rho(c_1),\,\dots,\,\rho(c_n)\}$. But since $\rho$ is dihedral, the conjugate of each 
$\rho(c_i)$ is $\rho(c_i)$ or $\rho(c_i)^{-1}$. It follows that $\tr\rho(a)$ belongs to a finite set depending 
only on the boundary traces $\tr\rho(c_i)$. Thus, we conclude that $\{\tr\rho(a)\,\mid\,a\subset\Sigma\text{ 
separating curve}\}$ is finite.

\begin{figure}[ht]
    \centering
    \includegraphics{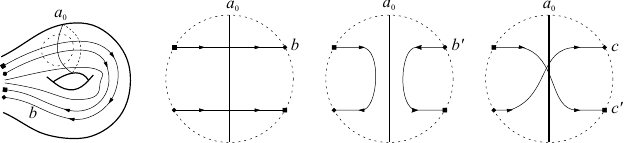}
    \caption{Building new loops out of old ones}
    \label{fig7}
\end{figure}

It remains to show that $\{\tr\rho(a)\,\mid\,a\,\subset\,\Sigma\text{ nonseparating curve}\}$ is finite. Let $a_0$ be a nonseparating curve in $\Sigma$ such that the restriction $\rho|(\Sigma\setminus a_0)$ is diagonal; such a curve exists since $\rho$ is special dihedral. Let $b$ be a nonseparating curve in $\Sigma$. Up to isotopy, we may assume that $b$ intersects $a$ only finitely many times. If $b$ does not intersect $a_0$, then $b\subset\Sigma\setminus a_0$ and hence $\tr\rho(b)$ can take only finitely many values as $\rho|(\Sigma\setminus a_0)$ is diagonal.

If $b$ intersects $a_0$ exactly once, then we must have $\tr\rho(b)\,=\,0$ by Lemma \ref{intlem}. Let us now
assume that $b$ intersects $a_0$ more than once. Let us choose a parametrization of $b$. Since $b$ is
nonseparating, there must be two neighboring points of intersection of $a_0$ and $b$ where the two segments
of $b$ have the same orientation, as in Figure \ref{fig7}. The operations as in Figure \ref{fig7} produce
for us a new simple closed curve $b'$ which is also nonseparating, as well as a pair $(c,\,c')$ of simple loops
in $(1,1)$-position on $\Sigma$. Note that $b'$ and $c,\,c'$ are constructed so that $b$ is homotopic to
$cc'$ (here the implicit base point is the intersection point of $c$ and $c'$), and $b'$ is homotopic
to $c'c^{-1}$. The fact that
$$\tr(AB)+\tr(AB^{-1})=\tr(A)\tr(B)$$
for any $A,\,B\,\in\,\SL_2(\C)$ shows, upon choosing $A\,=\,\rho(c)$ and $B\,=\,\rho(c')$, that we have
$$\tr\rho(b)=\tr\rho(c)\tr\rho(c')-\tr\rho(b').$$
By Lemma \ref{intlem}, we have $\tr\rho(c)\tr\rho(c')\,=\,0$, and therefore $\tr\rho(b)\,=\,-\tr\rho(b')$. Note
furthermore that $b'$ intersects $a_0$ in a smaller number of points than $b$ does. Applying induction
on the number of intersection points, we thus conclude that
$$\{\tr\rho(a)\,\mid\,a\subset\Sigma\text{ nonseparating curve}\}$$
is finite. This completes the proof that special dihedral representations have finite mapping class group 
orbits in $X$.
\end{proof}

\section{Analysis of Dehn twists}\label{sect:4}

Throughout this section, let $\Sigma$ be a surface of genus $g\geq1$ with $n\,\geq\,0$ punctures. For convenience 
of exposition, we shall denote by $(*)$ the following condition on a representation 
$\pi_1(\Sigma)\,\to\,\SL_2(\C)$:

\begin{quotation}
$(*)$ The representation is semisimple, and moreover $\Sigma$ has genus at least $2$ or the representation
is not special dihedral.
\end{quotation}

Given $\rho\,:\,\pi_1(\Sigma)\,\to\,\SL_2(\C)$ satisfying $(*)$, Theorems A and B (respectively, Theorem C) 
state that $\rho$ has finite (respectively, bounded) image in $\SL_2(\C)$ if its mapping class group orbit in 
the character variety $X(\Sigma)$ is finite (respectively, bounded). Let us recall the following.

\begin{theorem}[{\cite[Theorem 1.2]{psw}}]\label{psw}
Let $\Sigma$ be a surface of positive genus $g\,\geq\,1$ with $n\,\geq\,0$ punctures. If a semisimple representation 
$\pi_1(\Sigma)\,\to\,\SL_2(\C)$ has finite monodromy along all simple loops on $\Sigma$, then it has finite image.
\end{theorem}

\begin{proposition}[{\cite[Lemma 2.2]{psw}}]\label{psw2} 
	Let $\Sigma$ be a surface of positive genus $g\,\geq\,1$ with $n\,\geq\,0$ punctures. If a semisimple
representation $\pi_1(\Sigma)\,\to\,\SL_2(\C)$ has elliptic or central monodromy along all simple loops on
$\Sigma$, then it is unitarizable.
\end{proposition}

Consequently, to prove Theorems A and B (respectively, Theorem C) it suffices to prove that a representation
$\rho\,:\,\pi_1(\Sigma)\,\to\,\SL_2(\C)$ with finite (respectively, bounded) mapping class group orbit in the
character variety has finite (respectively, elliptic or central) monodromy along all nondegenerate simple
closed curves on $\Sigma$. Let us divide up the curves into four types:
\begin{itemize}
\item {Type I.} Nonseparating essential curves.

\item {Type II.} Separating essential curves $a\,\subset\,\Sigma$ with each component of
$\Sigma|a$ (defined in section 2.3) having genus at least one.

\item {Type III.} Separating essential curves $a\subset\Sigma$ with one component of $\Sigma|a$ having genus zero.

\item {Type IV.} Boundary curves.
\end{itemize}
The purpose of this section is to show that, if $\rho\,:\,\pi_1(\Sigma)\,\to\,\SL_2(\C)$ is a representation satisfying 
the hypotheses $(*)$ mentioned above with finite (respectively, bounded) mapping class group orbit in 
$X(\Sigma)$, then $\rho$ must have finite (respectively, elliptic or central) monodromy along all curves of 
type I and II. We shall also show that $\rho$ has finite (respectively, elliptic or central) monodromy along 
all curves of type III provided the same holds for all curves of type IV. In particular, this is enough for us 
to prove the following special cases of Theorems A, B and C.

\begin{proposition}
\label{simple-finite}
Let $\Sigma$ be a surface of genus $g\,\geq\,1$ with $n\,\geq\,1$ punctures.
Suppose that $\rho\,:\,\pi_1(\Sigma)\,\to\,\SL_2(\C)$ is a semisimple representation with finite (respectively,
elliptic or central) local monodromy around the punctures, and finite (respectively, bounded) mapping
class group orbit in the character variety $X(\Sigma)$. Then one of the following holds:
\begin{enumerate}
\item $\rho$ is finite (respectively, unitarizable);

\item $g\,=\,1$ and $\rho$ is special dihedral up to conjugacy.
\end{enumerate}
\end{proposition}

The rest of this Section proves Proposition \ref{simple-finite} by dealing with curves of Types I, II and III.
We shall complete the proof of our main results in Section \ref{sect:5} by treating the curves of type IV.
\\

\noindent {\bf Type I.} Nonseparating essential curves.

\begin{lemma}
\label{type1}
Let $\rho\,:\,\pi_1(\Sigma)\,\to\,\SL_2(\C)$ be a representation satisfying $(*)$, and let
$a\,\subset\,\Sigma$ be a curve of type I. If $\rho$ has finite (respectively, bounded) orbit under
$Mod(\Sigma)$ in $X(\Sigma)$, then $\rho$ has finite (respectively, elliptic or central) monodromy along $a$.
\end{lemma}

\begin{proof}
Let us choose a base point $x_0$ on $\Sigma$ lying on $a$, and let $\alpha$ be a simple based loop 
parametrizing $a$. Suppose $\beta$ is another simple loop such that the pair $(\alpha,\,\beta)$ is in 
$(1,\,1)$-position, and let $b$ denote the underlying curve. For each integer $k\,\in\,\Z$, the loop $\alpha^k\beta$ 
is homotopic to a simple loop whose underlying curve is isotopic to $\tw_a^k(b)$. In particular, by our 
hypothesis the set $\{\tr\rho(\alpha^k\beta)\,\mid\,k\in\Z\}$ is a finite (respectively, bounded) subset of 
$\C$. Up to global conjugation of $\rho$ by an element of $\SL_2(\C)$, we may consider two cases.

(a) Suppose first that
	$$\rho(\alpha)=\begin{bmatrix}\lambda & 0 \\ 0 &\lambda^{-1}\end{bmatrix},\quad \lambda\in\C^\times.$$
	Let $\beta$ be a simple loop on $\Sigma$ such that $(\alpha,\beta)$ is in $(1,1)$-position, and write
$$\rho(\beta)=\begin{bmatrix} b_1 & b_2\\ b_3 & b_4\end{bmatrix}.$$
	For each $k\in\Z$, we have
	$$\tr\rho(\alpha^k\beta)=\tr\left(\begin{bmatrix}\lambda^k &0 \\ 0 & \lambda^{-k}\end{bmatrix}\begin{bmatrix} b_1 & b_2\\ b_3 & b_4\end{bmatrix}\right)=\lambda^kb_1+\lambda^{-k}b_4.$$
	The fact that $\{\tr\rho(\alpha^k\beta)\,\mid\,k\in\Z\}$ is a finite (respectively, bounded) subset of $\C$ then implies that $\lambda$ is a root of unity (respectively, has absolute value $1$) so that $\rho(\alpha)$ is torsion (respectively, elliptic or central), or that $b_1=b_4=0$. Since the argument applies to any loop $\beta$ such that $(\alpha,\beta)$ is in $(1,1)$-position, it remains only to consider the case where $\rho(\beta)$ has both diagonal entries zero for every such $\beta$. In this case, we see upon reflection that
the restriction $\rho|(\Sigma|a)$ must be diagonal. This shows that $\rho$ is special dihedral. Since $\rho$ satisfies $(*)$ this means that $\Sigma$ moreover has genus at least $2$, so $\Sigma|a$ has genus at least $1$ and it follows from Lemma \ref{redlemm} that $\rho(\alpha)$ is torsion (respectively, elliptic or central).

(b) Suppose that
	$$\rho(\alpha)=s\begin{bmatrix}1 & x \\ 0 &1\end{bmatrix}$$
	for some $s\in\{\pm1\}$ and $x\in\C$. Let us assume $s=+1$; the case $s=-1$ will follow similarly. Let $\beta$ be a simple loop on $\Sigma$ such that $(\alpha,\beta)$ is in $(1,1)$-position, and let us follow the notation for $\rho(\beta)$ from the previous case. For each $k\in\Z$, we have
	$$\tr\rho(\alpha^k\beta)=\tr\left(\begin{bmatrix}1 &kx \\ 0 & 1\end{bmatrix}\begin{bmatrix} b_1 & b_2\\ b_3 & b_4\end{bmatrix}\right)=b_1+b_4+kxb_3.$$
 The fact that $\{\tr\rho(\alpha^k\beta)\,\mid\,k\in\Z\}$ is a finite (respectively, bounded) subset of $\C$ then implies that $x=0$ or $b_3=0$. Since the argument applies to any loops $\beta$ such that $(\alpha,\beta)$ is in $(1,1)$-position, it remains only to consider the case where $\rho(\beta)$ is upper triangular for every such $\beta$. In this case, $\rho$ is upper triangular (hence diagonal). It follows from Lemma \ref{redlemm} that $\rho(\alpha)$ is torsion (respectively, elliptic or central).

The above arguments show that $\rho(\alpha)$ is torsion (respectively, elliptic or central) unless $g\,=\,1$ and 
$\rho$ is special dihedral, as desired.
\end{proof}

\noindent {\bf Types II and III.} Separating essential curves.

\begin{lemma}
\label{type2}
Let $\rho\,:\,\pi_1(\Sigma)\,\to\,\SL_2(\C)$ be a representation satisfying $(*)$, and let $a\,\subset\,
\Sigma$ be a curve of type II. If $\rho$ has finite (respectively, bounded) orbit under $Mod(\Sigma)$ in
$X(\Sigma)$, then $\rho$ has finite (respectively, elliptic or central) monodromy along $a$.
\end{lemma}

\begin{lemma}
\label{type3}
Let $\rho\,:\,\pi_1(\Sigma)\,\to\,\SL_2(\C)$ be a representation satisfying $(*)$, and let $a\,\subset\,
\Sigma$ be a curve of type III. If $\rho$ has finite (respectively, bounded) orbit  under $Mod(\Sigma)$
in $X(\Sigma)$, and moreover $\rho$ has finite (respectively, elliptic or central) monodromy along
all punctures of $\Sigma$, then $\rho$ has finite (respectively, elliptic or central) monodromy along $a$.
\end{lemma}

\begin{proof}[Proof of Lemma \ref{type2} and Lemma \ref{type3}]
Let us choose a base point $x_0$ on $\Sigma$ lying on $a$, and let $\alpha$ be a simple based loop 
parametrizing $a$. Let $\Sigma_1$ and $\Sigma_2$ be the connected components of $\Sigma|a$, and lift the base 
point on $\Sigma$ to $\Sigma_1$ and $\Sigma_2$. Suppose $\beta\in\pi_1(\Sigma_1)$ and 
$\gamma\in\pi_1(\Sigma_2)$ are simple loops such that their product $\beta\gamma$ on $\Sigma$ is homotopic to a 
simple loop (with underlying curve denoted $d$, say) transversely intersecting $a$ exactly twice. For 
convenience, in this paragraph we shall call such a pair $(\beta,\gamma)$ \emph{good}. The loop 
$\beta\alpha^k\gamma\alpha^{-k}$ for each $k\in\Z$ is freely homotopic to a simple loop whose underlying curve 
belongs to the orbit $\langle\tw_a\rangle\cdot d$ (all curves considered up to isotopy). In particular, by our 
hypothesis the set $\{\tr\rho(\beta\alpha^k\gamma\alpha^{-k})\,\mid\,k\in\Z\}$ is a finite (respectively, 
bounded) subset of $\C$. Up to global conjugation of $\rho$ by an element of $\SL_2(\C)$, we may consider two 
cases.

(a) Suppose first that
	$$\rho(\alpha)=\begin{bmatrix}\lambda & 0 \\ 0 &\lambda^{-1}\end{bmatrix},\quad \lambda\in\C^\times.$$
	Let $(\beta,\gamma)$ be a good pair, and let us write
$$\rho(\beta)=\begin{bmatrix} b_1 & b_2\\ b_3 & b_4\end{bmatrix}\quad\text{and}\quad \rho(\gamma)=\begin{bmatrix} c_1 & c_2\\ c_3 & c_4\end{bmatrix}.$$
	For each $k\in\Z$, we have
	\begin{align*}\tr\rho(\beta\alpha^k\gamma\alpha^{-k})&=\tr\left(\begin{bmatrix}b_1 & b_2\\ b_3 & b_4\end{bmatrix}\begin{bmatrix}\lambda^k &0 \\ 0 & \lambda^{-k}\end{bmatrix}\begin{bmatrix}c_1 & c_2\\ c_3 & c_4\end{bmatrix}\begin{bmatrix}\lambda^{-k} &0 \\ 0 & \lambda^{k}\end{bmatrix}\right)\\
	&=b_1c_1+\lambda^{-2k}b_2c_3+\lambda^{2k}c_2b_3+b_4c_4.
	\end{align*}
The fact that $\{\tr\rho(\beta\alpha^k\gamma\alpha^{-k})\,\mid\,k\in\Z\}$ is a finite (respectively, bounded) subset of $\C$ then implies that $\lambda$ is a root of unity (respectively, has absolute value 1) so that $\rho(\alpha)$ is torsion (respectively, elliptic or central), or that $b_2c_3=c_2b_3=0$. If the former happens, then we are done.

Suppose that the latter happens, and that at least one of $b_2$, $b_3$, $c_2$, $c_3$ is nonzero. We shall 
assume $b_2\,\neq\,0$; the other cases will follow similarly. As $b_2c_3\,=\,0$, we must have $c_3=0$. Applying the 
same argument with $\gamma$ replaced by any simple loop $\gamma'\,\in\,\Sigma_2$ such that $(\beta,\,\gamma')$ is 
good, we are reduced to the case where $\rho|\Sigma_2$ is upper triangular. If $\rho|\Sigma_1$ is also upper 
triangular, then $\rho$ must be upper triangular (whence diagonal), and from Lemma \ref{redlemm} $\rho(\alpha)$ 
is torsion (respectively, elliptic or central). So suppose there is a simple loop $\beta'\in\pi_1(\Sigma)$ such 
that $\rho(\beta')$ is not upper triangular. By repeating the above argument with $\beta$ replaced by $\beta'$, 
we are reduced to the case where $\rho|\Sigma_2$ must be diagonal. If $\Sigma_2$ has genus at least $1$, then 
from Lemma \ref{redlemm} it follows that $\rho(\alpha)$ is torsion (respectively, is elliptic or central). If 
$\Sigma_2$ has genus $0$, then the fact that $\rho$ has finite (respectively, elliptic or central) local 
monodromy along the punctures implies that $\rho(\alpha)$ has finite (respectively, elliptic or central) 
monodromy.

It remains to consider the case where $b_2\,=\,b_3\,=\,c_2\,=\,c_3\,=\,0$. Running through all the good pairs $(\beta,\gamma)$ 
and repeating the above argument, we are left with the case where $\rho$ is diagonal; Lemma \ref{redlemm} then 
implies that $\rho(\alpha)$ is torsion (respectively, elliptic or central).

The second case:
	
(b) Suppose that
	$$\rho(\alpha)=s\begin{bmatrix}1 & x \\ 0 &1\end{bmatrix}$$
	for some $s\in\{\pm1\}$ and $x\in\C$. We assume $s=+1$; the case $s=-1$ will follow similarly. Let $(\beta,\gamma)$ be a good pair, and let us follow the notation for $\rho(\beta)$ and $\rho(\gamma)$ from the previous case. For each $k\in\Z$, we have
	\begin{align*}&\tr\rho(\beta\alpha^k\gamma\alpha^{-k})\\
	&=\tr\left(\begin{bmatrix}b_1 & b_2\\ b_3 & b_4\end{bmatrix}\begin{bmatrix}1 &kx \\ 0 & 1\end{bmatrix}\begin{bmatrix}c_1 & c_2\\ c_3 & c_4\end{bmatrix}\begin{bmatrix}1 &-kx \\ 0 & 1\end{bmatrix}\right)\\
	&=b_1c_1+b_2c_3+b_3c_2+b_4c_4+(b_1c_3-b_3c_1-b_4c_3+b_3c_4)kx-b_3c_3k^2x^2.
	\end{align*}
	The fact that $\{\tr\rho(\beta\alpha^k\gamma\alpha^{-k})\,\mid\,k\in\Z\}$ is a finite (respectively, bounded) subset of $\C$ then implies that $x=0$ or $b_3c_3=((b_1-b_4)c_3-b_3(c_1-c_4))=0$. If the former happens, then we are done.
	
	Suppose now that $x\neq0$ and that the latter happens, and that $b_3\neq0$ or $c_3\neq0$. We shall consider the case $b_3\neq0$; the other case will follow similarly. Note then we must then have $c_3=0$ and $c_1=c_4$. Applying the same argument with $\gamma$ replaced by any simple loop $\gamma'\in\Sigma_2$ such that $(\beta,\gamma')$ is good, we conclude that $\rho|\Sigma_2$ is upper triangular with image consisting of parabolic elements in $\SL_2(\C)$. If $\Sigma_2$ has genus at least one, then by considering $\tr\rho(\beta\gamma')$ for good pairs $(\beta,\gamma')$ with $\gamma'$ nonseparating we see that in fact $\rho|\Sigma_2$ must have image in $\{\pm\mathbf 1\}$. If $\Sigma_2$ has genus zero, then by our hypothesis on $\rho$ we again conclude that $\rho|\Sigma_2$ must have image in $\{\pm\mathbf 1\}$. \emph{A fortiori}, $\rho(\alpha)$ is central in both cases.
		
	It only remains to consider the case where $b_3=c_3=0$. Running through all the possible good pairs $(\beta,\gamma)$ and repeating the above argument, we are left with the case where $\rho$ is upper triangular (whence diagonal); then $\rho(\alpha)$ is torsion (respectively, elliptic or central). This completes the proof.
\end{proof}
 
\section{Proof of the main results}\label{sect:5}

The goal of this section is to complete the proof of Theorems A, B, and C. This section is organized as 
follows. In Section \ref{sec-irred} we record an irreducibility criterion for $\SL_2(\C)$-representations of 
positive genus surface groups, to be used in subsequent parts of this section. In Section \ref{sec-g1n2}, we 
establish Theorem C for surfaces of genus one with at most two punctures. In Sections \ref{sec:5.3}, 
\ref{sec:5.4}, and \ref{sec:5.5} we complete the proofs of our main theorems. Finally, in Section \ref{sec:5.6} 
we give an alternative proof of Theorem A for closed surfaces using \cite{gkm,px}.

\subsection{An irreducibility criterion}\label{sec-irred}
\begin{lemma}
\label{loops}
Let $(a,\,b,\,c)$ be a sequence of loops on a surface in $(1,\,2)$-position. The following pairs are in
$(1,\,1)$-position:
$$(a,\,b),\,(a,\,bc),\,(ca,\,b),\,(ab,\,bc),\,(ca,\,cb),\,(ac,\,bc),\,(ca,\,ab).$$
\end{lemma}

\begin{proof}
This is seen by drawing the corresponding loop configurations of homotopic clean sequences. See
Figure \ref{fig5}, noting that any segments not passing the central base point are not to be considered
as part of each loop configuration.
\end{proof}

\begin{figure}[ht]
    \centering
    \includegraphics{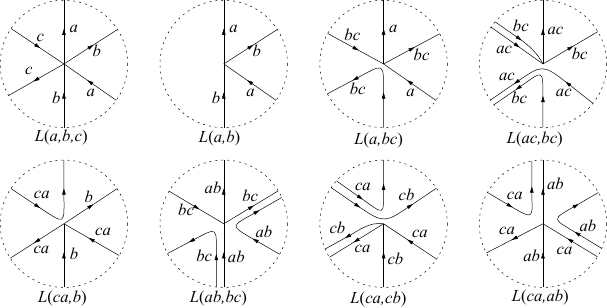}
    \caption{Loop configurations for Lemma \ref{loops}}
    \label{fig5}
\end{figure}

Let $\Sigma$ be a surface of genus $g\geq1$ with $n\,\geq\,0$ punctures. Given a pair $(a,\,b)$ of based loops in 
$(1,\,1)$-position on $\Sigma$, there is an embedding $\Sigma(a,\,b)\,\subset\,\Sigma$ of a surface of genus $1$ with 
$1$ boundary curve, i.e.,~a one-holed torus. Up to isotopy, every embedding of a one-holed torus is of the form 
$\Sigma(a,\,b)$ for some choice of $(a,\,b)$. The notion of loop configuration facilitates the proof of the 
following result, which will be used in the proof of our main theorems but may be of independent interest. (See 
\cite{coopermanning} for a proof when $n\,=\,0$.)

\begin{proposition}
\label{red}
Let $\Sigma$ be a surface of genus $g\,\geq\,1$ with $n\,\geq\,0$ punctures. A representation
$\rho\,:\,\pi_1(\Sigma)\,\to\,\SL_2(\C)$ is irreducible if and only if there is a one-holed torus subsurface
$\Sigma'\,\subset\,\Sigma$ such that the restriction $\rho|\Sigma'$ is irreducible.
\end{proposition}

\begin{proof}
The \emph{if} direction is clear. To prove the converse, let us begin by fixing a representation $\rho\,:\,
\pi_1(\Sigma)\,\to\,\SL_2(\C)$ whose restriction to every one-holed torus subsurface is reducible. In this proof, given
$a\,\in\,\pi_1(\Sigma)$ we shall also denote by $a$ the matrix $\rho(a)\,\in\,\SL_2(\C)$ for simplicity. The
statement that
the restriction $\rho|\Sigma(a,b)$ is reducible for an embedding $\Sigma(a,\,b)\,\subset\, \Sigma$
associated to a pair $(a,\,b)$ of loops in $(1,\,1)$-position is equivalent to saying that the pair $(\rho(a),\,
\rho(b))$ of matrices in $\SL_2(\C)$ has a common eigenvector in $\C^2$.

Throughout, we shall be using the following observation: if $a\,\in\,\SL_2(\C)\setminus\{\pm\mathbf 1\}$, and if 
$x,\,y,\,z\,\in\,\C^2$ are eigenvectors of $a$, then at least two of them are proportional; in notation, $x\,\sim\, y$, 
$x\,\sim\, z$, or $y\,\sim\, z$. First, we prove the following claim.

\begin{nclaim}
Any triple $(a,\,b,\,c)$ of loops on $\Sigma$ in $(1,\,2)$-position has a common eigenvector under the
representation $\rho$.
\end{nclaim}

\begin{proof}[Proof of claim]
Let $(a,\,b,\,c)$ be in $(1,\,2)$-position. Each of the pairs
$$(a,\,b),\,(a,\,bc),\,(ca,\,b),\,(ab,\,bc),\,(ca,\,cb),\,(ac,\,bc),\,(ca,\,ab)$$
is in $(1,\,1)$-position by Lemma \ref{loops}, and has a common eigenvector by our hypothesis on $\rho$. If any
of the elements $a,\,b,\,ab,\,ca,\,bc$ are $\pm\mathbf 1$ then $(a,\,b,\,c)$ clearly has a common eigenvector, so we will assume otherwise in what follows. The two steps below run in parallel.

\begin{enumerate}
\item Let $x$ be a common eigenvector of $(ca,\,b)$, and let $y$ be a common eigenvector of $(ca,\,cb)$. If $x\,\sim 
\,y$ then $x$ is a common eigenvector for $(a,\,b,\,c)$, so let us assume otherwise. In particular, $\{x,\,y\}$ is an 
eigenbasis for $ca$. Since $(ca,\,ab)$ has a common eigenvector, it follows that $ab$ has $x$ or $y$ as an 
eigenvector. If $ab$ has $x$ as an eigenvector, then $x$ is a common eigenvector for $(a,\,b,\,c)$. So assume $ab$ 
has $y$ as an eigenvector.

Note that $y$ is then also an eigenvector of $a^2=ab(cb)^{-1}ca$. We may assume that $a^2=-\mathbf 1$, since otherwise $y$ is also an eigenvector of $a$, and hence is a common eigenvector of $(a,b,c)$.

\item Let $x'$ be a common eigenvector of $(a,\,bc)$, and let $y'$ be a common eigenvector of $(ac,\,bc)$. If
$x'\,\sim\, y'$ then $x'$ is a common eigenvector for $(a,\,b,\,c)$, so let us assume otherwise. In particular,
$\{x',\,y'\}$ is an eigenbasis for $bc$. Since $(ab,\,bc)$ has a common eigenvector, it follows that $ab$ has
$x'$ or $y'$ as an eigenvector. If $ab$ has $x'$ as an eigenvector, then $x'$ is a common eigenvector for $(a,\,b,
\,c)$. So assume $ab$ has $y'$ has an eigenvector.

Note that $y'$ is then also an eigenvector of $b^2=bc(ac)^{-1}ab$. We may assume that $b^2=-\mathbf 1$, since otherwise $y'$ is also an eigenvector of $b$, and hence is a common eigenvector of $(a,b,c)$.

\end{enumerate}
Let $z$ be a common eigenvector of $(a,b)$. Since $a^2=b^2=-\mathbf 1$ by the above, the eigenvalue of $ab$ for the eigenvector $z$ is $\pm1$. Since $ab\neq\pm\mathbf1$ by assumption, it follows that $z$ must be a unique eigenvector of $ab$, and hence $z\sim y\sim y'$. Then $z$ is a common eigenvector of $(a,b,c)$. This completes the proof of our claim.
\end{proof}

To prove the proposition, we use the following inductive argument. First, let us assume that the image of 
$\rho$ is not contained in $\{\pm\mathbf 1\}$, since otherwise the proposition is clear. We thus can, and will, 
choose a nonseparating loop $a_1$ such that $a_1\,\neq\,\pm\mathbf 1$. Let us complete $a_1$ to an optimal sequence
$$(a_1,\,a_2,\,\cdots,\,a_{2g-1},\,a_{2g},\,a_{2g+1},\,\cdots,\,a_{2g+n})$$
of generators of $\pi_1(\Sigma)$. We show that $(a_1,\,\cdots,\,a_{2g+n})$ has a common eigenvector. For
simplicity of arguments we may assume that at least one element in each pair
$$(a_1,\,a_2),\,\cdots,\,(a_{2g-1},\,a_{2g})$$
is not equal to $\pm\mathbf 1$; for if some pair is of the form $(a,\,b)$ with $a,\,b\,\in\,\{\pm\mathbf 1\}$ we
may simply skip over that pair in the considerations below.

If $(g,\,n)\,=\,(1,\,0)$, then we are done since every representation $\pi_1(\Sigma)\,\to\,\SL_2(\C)$
is abelian. We thus assume $(g,\,n)\,\neq\,(1,\,0)$. By our claim, $(a_1,\,a_2,\,a_3)$ has a common
eigenvector. Assume next that $4\,\leq\, k\,\leq\, 2g+n$ and $(a_1,\,\cdots,\,a_{k-1})$ has a common eigenvector
$x\,\in\,\C^2$. We show that $(a_1,\,\cdots,\,a_k)$ has a common eigenvector. We consider the following cases.
\begin{enumerate}
	\item[\textup{(1)}] $k\,=\,5,\,7,\,\cdots,\,2g-1$, or $k\,=\,2g+1,\,2g+2,\,\cdots,\,2g+n$. The triple
$(a_1,\,a_2,\,a_k)$ is in $(1,\,2)$-position, and hence has a common eigenvector $y$ by our claim.
	
Given any $3\,\leq\, j\,<\,k$, note that $(a_1,\,a_2a_j,\,a_k)$ is in $(1,\,2)$-position and hence has a common
eigenvector by our claim, say $y_j$. Since $a_1\,\neq\,\pm\mathbf 1$ by assumption, we must have
	$$x\,\sim\, y,\,\quad x\,\sim\, y_j,\quad\text{or}\quad y_j\,\sim\, y.$$
	If one of the first two occurs, then we conclude that $(a_1,\,\cdots,\,a_k)$ has a common eigenvector
$x$, as required. Thus, we are left with the case $y_j\,\sim\, y$ for every $3\,\leq\, j\,<\,k$. But this implies
that $(a_1,\,a_2,\,a_2a_j,\,a_k)$ has a common eigenvector $y$, which is thus shared also by $a_j$. Thus, $y$ is
a common eigenvector of $(a_1,\,\cdots,\,a_k)$, as desired.
	
	\item[\textup{(2)}] $k\,=\,4,\,6,\,\cdots,\,2g$. First, consider the case where $a_{k-1}\,=\,
\pm\mathbf 1$. It then suffices to show that the sequence $(a_1,\,\cdots,\,a_{k-2},a_k)$ has a common
eigenvector. This can be shown by repeating the argument the previous case (1) above. Thus, we may assume that
$a_{k-1}\,\neq\,\pm\mathbf 1$. Note that, for each integer $m$ with $2m\,<\,k$, by our claim we have:
	\begin{itemize}
		\item a common eigenvector $w_m$ of $(a_{k-1},\,a_k,\, a_{2m})$, and
		\item a common eigenvector $z_m$ of $(a_{k-1},\,a_k,\, a_{2m-1})$.
	\end{itemize}
	Since $a_{k-1}\neq\pm\mathbf 1$, we must have
	$$x\sim w_m,\quad x\sim z_m,\quad\text{or}\quad w_m\sim z_m.$$
	If one of the first two occurs, then we conclude that $(a_1,\cdots,a_k)$ has a common eigenvector $x$, and we are done. Thus, we are left with the case where $w_m\sim z_m$ for every $2m<k$. Note in this case that $w_m$ is a common eigenvector of $(a_{2m-1},a_{2m},a_{k-1},a_k)$. Comparing the vectors $x,w_m$, and $w_{m'}$ for different $m,m'$, we are in turn reduced to the case $w_m\sim w_{m'}$ for all $m,m'$, in which case $(a_1,\cdots,a_k)$ has a common eigenvector $w_1$, as desired.
\end{enumerate}
Thus, $(a_1,\cdots,a_k)$ has a common eigenvector. This completes the induction, and shows that $(a_1,\cdots,a_{2g+n})$ has a common eigenvector, proving the proposition.
\end{proof}

\subsection{Genus $1$ with $1$ or $2$ punctures}\label{sec-g1n2}

We now give a separate proof of Theorem C for surfaces of genus one with one or two punctures. Let 
$\Sigma_{1,1}$ denote a surface of genus one with one puncture. We begin with the following:

\begin{lemma}\label{lem-s11}
Let $\rho \,: \,\pi_1(\Sigma_{1,1}) \,\to\, \SL_2(\C)$ be a semisimple representation such that each nonseparating
simple loop in $\Sigma_{1,1}$ maps to an elliptic or central element of $\SL_2(\C)$.
Then $\rho$ is unitarizable, i.e.,\ $\rho (\pi_1(\Sigma_{1,1}))$ has a global fixed point in $\H^3$.
\end{lemma}

\begin{proof}
This is known in the literature; we cite two sources below.

First, the lemma can be obtained as an immediate consequence of \cite[Theorem 1.2]{twz}. We sketch the 
connection of the setup in the current paper with that in \cite{twz}. Since every simple closed non-peripheral 
curve in $\Sigma_{1,1}$ maps to an elliptic element of $\SL_2(\C)$, it follows that we have 
$\tr\rho(a)\,\in\,[-2,\,2]$ for every $a$ in the curve complex of $\Sigma_{1,1}$. Hence, in the terminology of 
\cite{twz}, the set of end-invariants is given by the set of all projective measured laminations on 
$\Sigma_{1,1}$. It now follows from \cite[Theorem 1.2]{twz} that $\rho (\pi_1(\Sigma_{1,1}))$ is either 
unitarizable or dihedral. We finish by observing that a dihedral representation $\rho\,:\,\pi_1(\Sigma)\,\to\, 
D_\infty\,\subset\,\SL_2(\C)$ sending every simple nonseparating simple loop to an elliptic or central element is 
unitarizable.

Second, a direct proof of what is essentially the contrapositive, based on an explicit presentation of the 
character variety $X(\Sigma_{1,1})$ (see the Appendix), is given as the Algebraic Lemma in \cite[Section 
1.4.2]{dm}.
\end{proof}

We now turn to the twice-punctured torus $\Sigma_{1,2}$. We start with the following suggestive presentation of 
its fundamental group: $$\pi_1(\Sigma_{1,2}) \,=\,\langle u, x, y, p_1, p_2 \mid uxy = p_1, \, uyx = p_2 \rangle$$ 
where $p_1, \, p_2$ denote loops around the two different punctures of $\Sigma_{1,2}$. See Figure \ref{figs12} 
below and section 5.3 of \cite{goldman2}.

\begin{figure}[ht]
	\centering
	\includegraphics[scale=0.3]{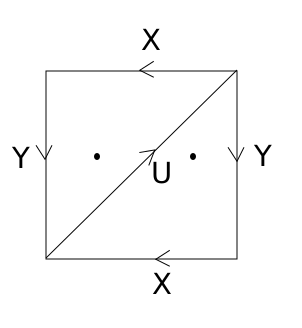}
	\caption{Preferred generators for $\Sigma_{1,2}$}
	\label{figs12}
\end{figure}

\begin{proposition}\label{prop-s12}
Let $\rho\,:\, \pi_1(\Sigma_{1,2}) \,\to\, \SL_2(\C)$ be a semisimple representation such that each nonseparating
simple loop in $\Sigma_{1,2}$ maps to an elliptic or central element of $\SL_2(\C)$.
Then $\rho$ is unitarizable.
\end{proposition}

Combining with Lemma \ref{type1}, we note that Lemma \ref{lem-s11} and Proposition \ref{prop-s12} prove Theorem 
C for surfaces of genus one with at most two punctures. Our proof of Proposition \ref{prop-s12} shall rely on 
the following observation. In what follows, $R_L$ denotes a hyperbolic reflection on the totally-geodesic plane 
$L$ in $\H^3$.

\begin{lemma}\label{lem-plane}
Let $A$ and $B$ be two non-trivial elliptic rotations with distinct but intersecting axes lying on a common 
plane $K$.  Then $AB$ is an elliptic rotation that has axis lying on a plane $Q$ that is at equal angles from 
$K$ and the plane $K^\prime$ obtained as the image of $K$ under $A$.
\end{lemma}

\begin{proof}
We note that $A$ is a composition of two hyperbolic reflections, that is,$ A \,=\, R_Q \circ R_K$ where $Q$ is a 
totally-geodesic plane containing the axis of $A$, and is at half the rotation angle (of $A$) from $K$. 
Similarly, $B\, =\, R_K \circ R_S$, where $S$ is a totally-geodesic plane that is at half the rotation angle of $B$ 
from $K$. Hence the composition $AB \,=\, R_Q \circ R_S$, and its axis is the intersection of the planes $Q$ and 
$S$. In particular, the axis lies on the plane $Q$, proving the lemma.
\end{proof}

\begin{proof}[Proof of Proposition \ref{prop-s12}]
Suppose that $\rho\,:\, \pi_1(\Sigma_{1,2})\,\to\, \SL_2(\C)$ is a representation satisfying the hypothesis of the 
proposition. We may assume that $\rho$ is irreducible, since otherwise the result is clear.

Moreover, we may assume that the restriction of $\rho$ to any one-holed torus subsurface of $\Sigma_{1,2}$ is 
semisimple. Indeed, otherwise, we may choose an optimal sequence $(a_1,\,b_1,\,c_1,\,c_2)$ of generators of 
$\pi_1(\Sigma)$ such that the restriction of $\rho$ to the one-holed torus $\Sigma'\,\subset\,\Sigma$ associated to 
the pair $(a_1,\,b_1)$ is upper triangular, with the boundary loop $c\,=\,c_1c_2$ having non-central parabolic 
monodromy. The irreducibility of $\rho$ then implies that $\rho(c_1)$ and $\rho(c_2)$ cannot be upper 
triangular. But then an argument as in part (b) of the proof of Lemma \ref{type3} shows that the restriction 
$\rho|\Sigma'$ is reducible and moreover $\rho(a_1)$ and $\rho(b_1)$ are parabolic, whence they must both be 
central and $\rho(c)$ is also central; a contradiction.

Thus, in what follows, $\rho\,:\, \pi_1(\Sigma_{1,2}) \,\to\, \SL_2(\C)$ is an irreducible representation satisfying 
the hypothesis of the proposition, with the property that its restriction to every one-holed torus subsurface 
is semisimple. We may further assume that none of the nonseparating simple loops have central monodromy. 
Indeed, suppose we have a nonseparating simple loop $u$ such that $\rho(u)$ is central. Completing $u$ to a 
sequence $(u,\,x,\,y)$ giving a preferred presentation of $\pi_1(\Sigma_{1,2})$ above, it then suffices to show 
that the restriction of $\rho$ to the subsurface $\Sigma(x,y)\,\subset\,\Sigma $ of genus one with one boundary 
curve is unitarizable, which is treated in Lemma \ref{lem-s11}.

Let us write $U \,=\, \rho (u)$, $X\, =\, \rho (x)$, and $Y\, =\, \rho (y)$ for the presentation of 
$\pi_1(\Sigma_{1,2})$ given above. By Lemma \ref{lem-s11}, the elements $U, X, Y$ have coplanar, pairwise 
intersecting axes. We call the common plane $P$. Let $p$ be the intersection point of the axes of $X$ and $Y$. 
It suffices to prove that $U$ fixes $p$ as well, that is, the axis of $U$ also passes through $p$. Assume that 
the axis of $U$ does not pass through $p$; we shall eventually reach a contradiction.

Consider the element $ U XY X\,=\,UT,$ say. Let $Z\,=\,X^{-1} Y X$ so that $T\, =\, X^2 Z$. Note that $T$ is 
elliptic as it fixes the intersection of the axes of $X, Y$.  Moreover, it can be easily checked using the loop 
configuration diagram that the curve represented by the element $uxyx$ is simple, closed, and essential.  Hence 
$UT$ is also an elliptic element. However, we also have:

\begin{claim}\label{claim-non-coplanar}
The axis of $T \,=\, XYX$ does not lie on the plane $P$  unless $X$ is a $\pi$--rotation. 
\end{claim}

\begin{proof}[Proof of claim] 
Note that since $Z \,=\, X^{-1}YX$,  the axis of $Z$ is the image of the axis of $Y$ under $X^{-1}$.  
Choose the plane $K$ to be the one containing the axes of $X$ and $Z$. 
Then 
$ K\,=\,X^{-1}(P)$, i.e.,\ the image of $K$ under $X$ is the plane $P$.
Since $X^2$ and $Z$ have the same axes as $X$ and $Z$ respectively,  Lemma \ref{lem-plane} shows that the
axis of $X^2Z$ will be on a plane that is at equal angles from $K$ and $P$.  In particular, the axis
does not lie on the plane $P$ unless $X^2$ is the identity element, that is, unless $X$ is a $\pi$--rotation.
\end{proof}

Since we have assumed that the axis of $U$ does not pass through $p$, the plane $P$ is the unique plane that 
contains both $p$ and the axis of $U$. Note that the axis of $T$ contains $p$; hence if it does not lie in the 
plane $P$, the axes of $U$ and $T$ cannot lie on a common plane. By Lemma 3.4.1 of \cite{gkm}, we have a 
contradiction to the fact that $UT$ is elliptic.  Hence, $X$ must be a $\pi$--rotation. Repeating the same 
argument for other presentations of $\pi_1(\Sigma_{1,2})$, we are thus reduced to the case where the monodromy 
trace of every nonseparating simple loop on $\Sigma_{1,2}$ under $\rho$ is $0$.

We claim that such $\rho$ must be dihedral up to conjugacy. Indeed, let $(a_1,\,b_1,\,c_1,\,c_2)$ be an optimal 
sequence of generators for $\pi_1(\Sigma_{1,2})$. Since
$$(\tr\rho(a_1),\,\tr\rho(b_1),\,\tr\rho(a_1b_1))\,=\,(0,\,0,\,0)$$
which determines the restriction of $\rho$ to $\langle a_1,\,b_1\rangle$ up to conjugacy, up to global
conjugation we may assume that
$$\rho(a_1)\,=\,\begin{bmatrix}i & 0 \\ 0 & -i\end{bmatrix}\quad\text{and}\quad\rho(b_1)=\begin{bmatrix}0 & 1\\ -1 & 0\end{bmatrix}.$$
Writing
$$\rho(c_1)\,=\,\begin{bmatrix}x & y\\ z & w\end{bmatrix},$$
we have
$$\rho(b_1c_1)\,=\,\begin{bmatrix}
z & w\\ -x & -y
\end{bmatrix}\quad\text{and}\quad
\rho(a_1b_1c_1)\,=\,\begin{bmatrix}
iz & iw\\ ix & iy
\end{bmatrix}.
$$
Since $\tr\rho(b_1c_2)\,=\,\tr\rho(a_1b_1c_1)\,=\,0$, the above expressions show that $y
\,=\,z\,=\,0$, i.e.,~$\rho(c_1)$ is diagonal. This shows that $\rho$ is dihedral up to conjugacy as desired. So
the hypothesis of the Proposition implies that $\rho$ is unitarizable, a contradiction.
\end{proof}

\subsection{Proof of Theorem A}\label{sec:5.3}
We restate and complete the proof of Theorem A.

\begin{theorema}
Let $\Sigma$ be an oriented surface of genus $g\,\geq\,2$ with $n\,\geq\,0$ punctures. A semisimple
representation $\rho\,:\,\pi_1(\Sigma)\,\to\,\slc$ has finite mapping class group orbit in the character
variety $X(\Sigma)$ if and only if $\rho$ is finite.
\end{theorema}

\begin{proof}
Suppose that $\rho\,:\,\pi_1(\Sigma)\,\to\,\SL_2(\C)$ is a semisimple representation with finite mapping class
group orbit in $X(\Sigma)$. By Lemma \ref{redlemm}, we are done if $\rho$ is reducible, so we may assume
that $\rho$ is irreducible. It suffices to show that $\rho$ has finite monodromy along every simple
closed curve around a puncture of $\Sigma$, for then we reduce to Proposition \ref{simple-finite}. By
Lemmas 4.4 and 4.5, we see that $\rho$ must have finite monodromy along any essential curve $a\,\subset\,
\Sigma$ which is either:
\begin{itemize}
	\item nonseparating, or
	\item separating with each component of $\Sigma|a\,=\,\Sigma_1\sqcup\Sigma_2$ having positive genus.
\end{itemize}
By Proposition \ref{red}, there is a one-holed torus subsurface $\Sigma'\,\subset\,\Sigma$ such that
$\rho|\Sigma'$ is irreducible. Let $c$ be a simple closed curve around a puncture of $\Sigma$. Let $\Sigma''
\,\subset\,\Sigma$ be a two-holed torus containing $\Sigma'$ and having $c$ as one of its boundary
curves; let $c'$ be the other boundary curve of $\Sigma''$. (Note here that each component of $\Sigma|c'$ has
positive genus, by design.) We shall prove that $\rho|\Sigma''$ is finite, so \emph{a fortiori} $\rho$ has
finite monodromy along $c$. For this, we follow the strategy of \cite{psw} below.

First, we shall show that the restriction of $\rho$ to $\Sigma''$ is unitarizable. This follows from 
Proposition \ref{prop-s12}, but for the benefit of the reader we also give another proof. First, we know from 
above that $\rho$ has finite monodromy along every essential curve of $\Sigma''$, as well as along the boundary 
curve $c'$. In particular, the restriction of $\rho$ to the one-holed torus $\Sigma'$ has image that is 
conjugate to a subgroup of $\SU(2)$. Let $A\,\subset\,\R$ be the $\Z$-algebra generated by the set of traces of 
$\rho$ along the essential curves of $\Sigma''$. By considering the preferred generators for $\pi_1(\Sigma'')$ 
introduced in Section \ref{sec-g1n2}, it follows from the trace relations given in Example \ref{exfree} 
(cf.~\cite[Section 5.3]{goldman2}) that $\tr\rho(c)$ satisfies a monic quadratic equation over the ring $A$, 
with the other root being $\tr\rho(c')$. Since $\tr\rho(c')\in\R$, it follows that $\tr\rho(c_1)\,\in\,\R$ as well. 
Applying Fact \ref{fact} to a sequence of optimal generators for $\pi_1(\Sigma'')$, we deduce that the 
character of $\rho|\Sigma''$ is real, and since $\rho|\Sigma''$ is semisimple its image is conjugate to a 
subgroup of $\SU(2)$ or $\SL_2(\R)$ (see e.g.~\cite[Proposition III.1.1]{ms}). The latter cannot occur, since 
otherwise the restriction of $\rho$ to $\Sigma'$ has image conjugate to a subgroup of $\SO(2)$, contradicting 
the fact that $\rho|\Sigma'$ is irreducible hence nonabelian. Thus, the restriction of $\rho$ to $\Sigma''$ has 
image conjugate to a subgroup of $\SU(2)$.

It also follows from the above analysis that the character of $\rho|\Sigma''$ takes values in the ring of 
algebraic integers in $\overline{\Q}$. In particular, we may assume without loss of generality that the image 
of $\rho|\Sigma''$ lies in $\SL_2(\overline{\Q})$. By considering conjugates of $\rho|\Sigma''$ by elements of 
the absolute Galois group $\Gal(\overline{\Q}/\Q)$ of $\Q$ and noting that the above analysis goes through for 
all of these conjugates, we conclude that the eigenvalues of monodromy of $\rho$ along $c$ are algebraic 
integers all of whose Galois conjugates have absolute value $1$. By Kronecker's theorem, it follows that the 
eigenvalues are roots of unity, i.e.,~$\rho$ has finite monodromy along $c$ (note that monodromy along $\rho$ 
cannot be unipotent). This is the desired result.
\end{proof}

\subsection{Proof of Theorem B}\label{sec:5.4}
We restate and complete the proof of Theorem B.

\begin{theoremb}
Let $\Sigma$ be an oriented surface of genus $1$ with $n\,\geq\,0$ punctures. A semisimple representation
$\rho\,:\,\pi_1(\Sigma)\,\to\,\SL_2(\C)$ has finite mapping class group orbit in the character variety
$X(\Sigma)$ if and only if $\rho$ is finite or special dihedral.
\end{theoremb}

\begin{proof}
Our proof proceeds as in the proof Theorem A, with minor modifications. Suppose that $\rho\,:\, 
\pi_1(\Sigma)\,\to\,\SL_2(\C)$ is a semisimple representation with finite mapping class group orbit in $X(\Sigma)$. 
By Lemma \ref{redlemm}, we are done if $\rho$ is reducible, so we may assume that $\rho$ is irreducible. We may 
also assume that $\rho$ is not special dihedral. It suffices to show that $\rho$ has finite monodromy along 
every simple closed curve around a puncture of $\Sigma$, for then we reduce to Proposition \ref{simple-finite}.

By Lemma \ref{type1}, $\rho$ has finite monodromy along any nonseparating essential curve $a\,\subset\,\Sigma$. 
Given any $c$ which is a simple closed curve around a puncture of $\Sigma$ or a separating essential curve on $\Sigma$, there is a two-holed torus subsurface of 
$\Sigma$ having $c$ as one of its boundary components. By the trace relations in Example \ref{exfree} 
(cf.~\cite[Section 5.3]{goldman2}), we see that $\tr\rho(c)$ is an algebraic integer. It follows that $\tr\rho(a)$ is an algebraic integer for every simple closed curve $a$ on $\Sigma$. This easily implies (see e.g.~\cite[Section 2.1]{psw}) that $\tr\rho(\alpha)$ is an algebraic integer for every $\alpha\in\pi_1(\Sigma)$. We may in particular 
assume that $\rho$ is a representation of $\pi_1(\Sigma)$ into $\SL_2( \overline{\Q})$.

Now, $\rho$ is unitarizable by Theorem C proved below, and in particular the monodromy eigenvalues of $\rho$ 
along any curve $c$ around a puncture of $\Sigma$ have absolute value $1$. Applying this observation to every 
conjugate of $\rho$ by an element of the absolute Galois group $\Gal(\overline{\Q}/\Q)$, we see that the 
eigenvalues of $\rho(c)$ for any curve around a puncture of $\Sigma$ are algebraic integers all of whose 
conjugates have absolute value $1$. By Kronecker's theorem, it follows that the eigenvalues are roots of unity, 
i.e.,~$\rho$ has finite monodromy along $c$ (note that monodromy along $\rho$ cannot be unipotent). This is the 
desired result.
\end{proof}

\subsection{Proof of Theorem C}\label{sec:5.5}
We restate and complete the proof of Theorem C.

\begin{theoremc}
Let $\Sigma$ be an oriented surface of genus $g\,\geq\,1$ with $n\,\geq\,0$ punctures. A semisimple representation
$\rho\,:\,\pi_1(\Sigma)\,\to\,\slc$ has bounded mapping class group orbit in the character variety $X(\Sigma)$ if
and only if:
\begin{enumerate}
\item $\rho$ is unitary up to conjugacy, or

\item $g\,=\,1$ and $\rho$ is special dihedral up to conjugacy.
\end{enumerate} 
\end{theoremc}

\begin{proof}
In the case where $\Sigma$ is a surface of genus $g\geq2$, a minor modification of the argument in the proof of 
Theorem A proves that $\rho$ has elliptic or central local monodromy around the punctures. Proposition 
\ref{simple-finite} then shows that $\rho$ is unitarizable, as desired.

Let us now assume that $\Sigma$ has genus $1$ with $n\,\geq\,0$ punctures. We know the case $n\,\leq\, 2$ of Theorem C 
by our work in Section \ref{sec-g1n2}; we shall deduce the general case from it. So suppose we have a 
semisimple representation $\rho\,:\,\pi_1(\Sigma)\,\to\,\SL_2(\C)$ whose monodromy is central or elliptic along every 
nonseparating curve on $\Sigma$. It is clear that $\rho$ is unitarizable if it is moreover reducible, so let us 
assume that $\rho$ is irreducible in what follows.

By Proposition \ref{red}, there is a one-holed torus subsurface $\Sigma'\,\subset\,\Sigma$ such that $\rho|\Sigma'$ 
is irreducible. For any two-holed torus $\Sigma''\,\subset\,\Sigma$ containing $\Sigma'$, Proposition 
\ref{prop-s12} shows that $\rho|\Sigma''$ is unitarizable. (\emph{A fortiori}, $\rho|\Sigma'$ is unitarizable.) In particular the monodromy of $\rho$ along each of the two boundary curves of such $\Sigma''$ is central or elliptic; and hence has real trace. 
By considering an optimal sequence of generators for $\pi_1(\Sigma')$ and extending it to an optimal sequence of generators for $\pi_1(\Sigma)$, and applying Fact \ref{fact}, we see that 
the coordinate ring of the character variety $X(\Sigma)$ is generated by the monodromy traces $\tr_a$ around a 
finite collection of curves $a$ each of which is either a nonseparating curve, the boundary curve of the one-holed torus $\Sigma'$, or a boundary curve of one of the two-holed tori $\Sigma''$ mentioned above. In particular, it follows that $\rho$ has real character, and therefore the image of $\rho$ is 
conjugate to a subgroup of $\SU(2)$ or $\SL_2(\R)$. We claim that the latter cannot occur. Indeed, if the 
image of $\rho$ lies in $\SL_2(\R)$, then unitarizability of $\rho|\Sigma'$ implies that the image of $\rho|\Sigma'$ is conjugate 
to a subgroup of $\SO(2)$, contradicting the irreducibility of $\rho|\Sigma'$. Thus, the image of $\rho$ must 
be conjugate to a subgroup of $\SU(2)$, as desired.
\end{proof}

\subsection{Alternative proof of Theorem A for closed surfaces}\label{sec:5.6}

We give a different short proof of Theorem A in the case where $\Sigma$ is a closed surface of genus $g\,\geq\,2$, 
using the following result of Gallo--Kapovich--Marden.

\begin{theorem}[{\cite[Section 3]{gkm}}]\label{gkm} Let $\S$ be a closed oriented surface of genus greater than one.
If  $\rho\,:\,\pi_1(\Sigma)\,\to\,\SL_2(\C)$ is non-elementary, then there exist simple loops $a,\, b$ on $\S$
such that 
\begin{enumerate}
\item the intersection number $i(a,\,b) \,=\, 1$,

\item The images $\rho(a), \,\rho(b)\,\in \,\PSL_2(\C)$ are loxodromic and distinct, and generate a Schottky group.
	\end{enumerate}
\end{theorem}

\begin{theorem}\label{rgg} Let $\S$ be a closed orientable surface of genus $g\,\geq\,2$.
Given a semisimple representation $\pi_1(\Sigma)\,\to\,\SL_2(\C)$, the orbit
of $\rho$ in $X(\Sigma)$, under the action of $\Mod(\Sigma)$, is finite if and only if the image
of $\rho$ is a finite group.
\end{theorem}

\begin{proof}
Suppose first that $\rho$ is non-elementary. By Theorem \ref{gkm}, there exist simple loops $a,\, b$ on $\S$
with $i(a,\,b)\, =\, 1$ such that $\rho(a), \,\rho(b)
\,\in\, \PSL_2(\C)$ are loxodromic and distinct. Since $\rho(a),\, \rho(b)\,\in\, \PSL_2(\C)$ are loxodromic,
$\tw_a^n (b)$ gives an infinite sequence of curves in $\Sigma$ with $\rho$-images $\rho(a)^n \rho(b)$ whose
translation length in $\HHH^3$ tends to infinity as $n\,\to\, \infty$, while $\tw_a^n (a)$ remains fixed. It
follows that the $\tw_a^n$-orbit of $\rho$ is infinite in $X(\Sigma)$ and hence so is the $\Mod(\Sigma)$-orbit.

Suppose that $\rho$ is elementary. In view of the results in Section \ref{sect:3}, it remains only to treat the 
case where $\rho$ is a representation whose image is a dense subgroup of $\SU(2)$ in the Euclidean topology. In 
this case, the main theorem of \cite{px} states that the $\Mod(\Sigma)$-orbit of $\rho$ is dense in 
$\Hom(\pi_1(\Sigma),\,\SU(2))/\SU(2)$ in the Euclidean topology. Hence it is infinite.
\end{proof}

\section{Applications}\label{sect:6}

In this section, we collect applications of our results and methods developed in the previous sections. The 
following immediate corollary of Theorem A answers a question that was posed to us by Lubotzky:

\begin{corollary}\label{faithful}
Given a surface $\Sigma$ of genus $g\,\geq\,1$ with $n\geq0$ punctures, a faithful representation
$\rho\,:\,\pi_1(\Sigma)\,\to\, \SL_2(\C)$ (or $\PSL_2(\C)$) cannot have finite $\Mod(\Sigma)$-orbit in the
character variety.
\end{corollary}

\begin{proof}
It follows from Theorems A and B that if $\rho\,:\,\pi_1(\Sigma)\,\to\, \SL_2(\C)$  has a finite
$\Mod(\Sigma)$-orbit in $X(\Sigma)$, then $\rho$ cannot be faithful. The same holds when $\SL_2(\C)$ is
replaced by $\PSL_2(\C)$.
\end{proof}

Let $\Sigma$ be a closed surface of genus greater than one; we fix a a hyperbolic metric on $\Sigma$ such that 
$\Sigma\,=\, \HHH^2/\Gamma$ where $\Gamma$ is a Fuchsian group. Recall that for any semisimple representation 
$\rho\,:\,\Gamma\,\to\, \PSL_2(\C)$, there exists a $\rho$-equivariant harmonic map $\widetilde{h}
\,:\, \HHH^2\,\to\, \HHH^3$ from 
the universal cover of $\Sigma$ to the symmetric space for $\PSL_2(\C)$ (see, for example, \cite{Donaldson}). 
The \textit{equivariant energy} of this harmonic map is the energy of its restriction to a fundamental domain 
of the $\Gamma$-action on $\HHH^2$. The following answers a question due to Goldman.

\begin{theorem}\label{energy}
Fix a Riemann surface $\S$ of genus greater than one with $\Gamma \,=\, \pi_1(\S)$, and let $\rho
\,:\,\Gamma \,\to\, \PSL_2(\C)$ be an semisimple representation. Suppose that the equivariant energies
of the harmonic maps corresponding to the mapping class group orbit of $\rho$ in
$\Hom(\pi_1(\Sigma),\,\PSL_2(\C))\git\PSL_2(\C)$ is uniformly bounded. Then $\rho(\Gamma)$ fixes a point
of $\HHH^3$; in particular $\rho(\Gamma)$ can be conjugated to lie in $\PSU(2)$.
\end{theorem}

\begin{proof} 
If $\rho(\Gamma)$ is elementary, but not unitary, then $\rho(\Gamma)$ fixes a geodesic in $\HHH^3$, and thus 
the image of any equivariant harmonic map coincides with this geodesic. Let $a$ be a simple closed curve mapped 
to an infinite order hyperbolic element in $\rho(\Gamma)$. Let $b$ be a simple closed on $\Sigma$ with $i(a,\,b) 
\,>\, 0$. Then the translation length of $\rho(\tw_a^n b)$ increases to infinity as $n \to \infty$. Hence the 
$\Mod(\Sigma)-$orbit of $\rho$ cannot have bounded energy, since an uniform energy bound on the harmonic maps 
implies that they are uniformly Lipschitz.
	
Otherwise, suppose $\rho(\Gamma)$ is non-elementary. By Theorem \ref{gkm}, there exist simple closed curves $a, 
\,b$ on $S$ such that $i(a,\,b) \,=\, 1$ and $\rho(a),\, \rho(b) \,\in\, \PSL_2(\C)$ are loxodromic and 
distinct. It follows as in the proof of Theorem \ref{rgg} that the translation lengths in $\HHH^3$ of 
$\rho(\tw_a^n (b))$ tend to infinity, while $\tw_a^n (a)$ remains fixed. Hence the $\Z-$orbit $\rho 
\circ\tw_a^n$ of $\rho$ under $\langle \tw \rangle$ cannot have bounded energy. Thus the $\Mod(\Sigma)-$orbit 
of $\rho$ cannot have bounded energy, as before.
	
The possibility that remains is that $\rho(\Gamma)$ is unitary. 
\end{proof}

\appendix

\section{Case of once-punctured torus} \label{appendix-a}
In this appendix, we describe how the work of Dubrovin--Mazzocco \cite{dm} can be used to prove the once-punctured torus case of our Theorem B. We begin with the following well-known observation.

\begin{lemma}
\label{tracelem}
A pair $(a,b)$ of elements in $\SL_2(\C)$ has a common eigenvector in $\C^2$, or in other words lies in the standard Borel $B$ up to simultaneous conjugation, if and only if $\tr([a,b])=2$, where $[a,b]=aba^{-1}b^{-1}$.
\end{lemma}

Let $\Sigma$ be a surface of genus one with one puncture. Let $(a,b,c)$ be an optimal sequence of generators for $\pi_1(\Sigma)$, as defined in Example \ref{exgen}. Let $X=X(\Sigma)$ be the character variety of $\Sigma$. Note that $\pi_1(\Sigma)$ is freely generated by $a$ and $b$. The trace functions on $X(\Sigma)$ furnish an isomorphism
$$(x_1,x_2,x_3)=(\tr_{a},\tr_{b},\tr_{ab}):X(\Sigma)\xrightarrow{\sim}\A^3$$
by Fricke (see \cite{goldman2} for details). Observing that $\tr(\mathbf1)=2$ for the identity $\mathbf1\in\SL_2(\C)$ and that $\tr(A)\tr(B)=\tr(AB)+\tr(AB^{-1})$ for any $A,B\in\SL_2(\C)$, we have
\begin{align*}
\tr_c&=\tr_{aba^{-1}b^{-1}}=\tr_{aba^{-1}}\tr_{b^{-1}}-\tr_{aba^{-1}b}\\
&=\tr_{b}^2-\tr_{ab}\tr_{a^{-1}b}+\tr_{aa}=\tr_{b}^2-\tr_{ab}(\tr_{a^{-1}}\tr_{b}-\tr_{ab})+\tr_{a}^2-\tr_{\mathbf1}\\
&=\tr_{a}^2+\tr_{b}^2+\tr_{ab}^2-\tr_{a}\tr_{b}\tr_{ab}-2.
\end{align*}
In particular, under the identification $(x_1,x_2,x_3)=(\tr_{a},\tr_{b},\tr_{ab})$ of the coordinate functions above and Lemma \ref{tracelem}, we see that the locus of reducible representations in $X(\Sigma)=\A^3$ is the cubic algebraic surface cut out by the equation
$$x_1^2+x_2^2+x_3^2-x_1x_2x_3-2=2.$$
The mapping class group $\Mod(\Sigma)$ acts on $X(\Sigma)$ via polynomial transformations.  For convenience, we shall denote the isotopy classes of simple closed curves lying in the free homotopy classes of $a$, $b$, and $ab$ by the same letters. We have the following descriptions of the associated Dehn twist actions.

\begin{lemma}
\label{dehnt}
The Dehn twist actions $\tw_a$, $\tw_b$, and $\tw_{ab}$ on $X(\Sigma)$ are given by
\begin{align*}
\tw_{a}^*&:(x_1,x_2,x_3)\mapsto (x_1,x_3,x_1x_3-x_2),\\
\tw_{b}^*&:(x_1,x_2,x_3)\mapsto (x_1x_2-x_3,x_2,x_1),\\
\tw_{ab}^*&:(x_1,x_2,x_3)\mapsto(x_2,x_2x_3-x_1,x_3).
\end{align*}
in terms of the above coordinates.
\end{lemma}

\begin{proof}
Note that $\tw_a(a)$ has the homotopy class of $\alpha$, $\tw_a(b)$ has the homotopy class of $\alpha\beta$, and $\tw_a(ab)$ has the homotopy class of $\alpha\alpha\beta$. Noting that $\tr_{\alpha\alpha\beta}=\tr_{\alpha}\tr_{\alpha\beta}-\tr_{\beta}$, we obtain the desired expression for $\tw_a^*$. The other Dehn twists are similar.
\end{proof}

Let $\Pi$ be the group of polynomial automorphisms of $\A^3$ generated by $\tw_{a}^*$, $\tw_{b}^*$, and $\tw_{ab}^*$. It is precisely the image of the mapping class group $\Mod(\Sigma)$ in the group of polynomial automorphisms of $X(\Sigma)=\A^3$. Let $\Pi'$ be the group generated by $\Pi$ together with the following transformations:
\begin{align*}
\sigma_{12}:(x_1,x_2,x_3)\mapsto(-x_1,-x_2,x_3),\\
\sigma_{23}:(x_1,x_2,x_3)\mapsto(x_1,-x_2,-x_3),\\
\sigma_{13}:(x_1,x_2,x_3)\mapsto(-x_1,x_2,-x_3).
\end{align*}
It is easy to see that $[\Pi':\Pi]<\infty$. Hence, a point in $\A^3$ has finite $\Pi$-orbit if and only if it has finite $\Pi'$-orbit. Now, the group $\Pi'$ contains a group generated by transformations
\begin{align*}
\beta_1=\sigma_{12}\tw_{ab}^*(\tw_{b}^*\tw_{a}^*)^{-1}&:(x_1,x_2,x_3)\mapsto(-x_1,x_3-x_1x_2,x_2),\\
\beta_2=\sigma_{23}\tw_{a}^*(\tw_{b}^*\tw_{a}^*)^{-1}&:(x_1,x_2,x_3)\mapsto(x_3,-x_2,x_1-x_2x_3).
\end{align*}
whose finite orbits in $\A^3$ were studied by Dubrovin-Mazzocco \cite[Theorem 1.6]{dm} in connection with algebraic solutions of special Painlev\'e VI equations. They defined a triple $(x_1,x_2,x_3)\in\A^3(\C)$ to be \emph{admissible} if it has at most one coordinate zero and $x_1^2+x_2^2+x_2^2-x_1x_2x_3-2\neq2$. It is easy to verify that the admissible points are precisely those which do not correspond to reducible or special dihedral representations. The result of \cite{dm} we shall use is the following.

\begin{theorem}[Dubrovin-Mazzocco]
\label{dubm}
The following is a complete set of representatives for the finite $\langle\beta_1,\beta_2\rangle$-orbits of admissible triples in $\A^3$:
\begin{align*}
&(0,-1,-1), (0,-1,-\sqrt{2}),(0,-1,-\varphi),(0,-1,-\varphi^{-1}),(0,-\varphi,-\varphi^{-1})
\end{align*}
where $\varphi=(1+\sqrt{5})/2$ is the golden ratio.
\end{theorem}

To deduce Theorem B in the once-punctured torus case from the above, we recall the following explicit description of the finite subgroups $B A_4$, $B S_4$, $B A_5$ of $\SL_2(\C)$. First, let us identify the group of unit quaternions
$$\Sp(1)=\{z=(a,b,c,d)=a+bi+cj+dk\in\H:|z|=a^2+b^2+c^2+d^2=1\}$$
as a subgroup of $\SL_2(\C)$ by the map
$$z=(a,b,c,d)\mapsto \begin{bmatrix}a+bi & c+di\\ -c+di & a-bi\end{bmatrix}.$$
Under the identification, the \emph{binary tetrahedral group} $BA_4$ is given by
$$B A_4=\{\pm1,\pm i,\pm j,\pm k,(\pm1\pm i\pm j\pm k)/2\}$$
with all sign combinations taken in the above. The \emph{binary octahedral group} $B S_4$ is the union of $B A_4$ with all quaternions obtained from $(\pm1,\pm1,0,0)/\sqrt 2$ by all permutations of coordinates and all sign combinations. The \emph{binary icosahedral group} $B A_5$ is the union of $B A_4$ with all quaternions obtained from $(0,\pm1,\pm\varphi^{-1},\pm\varphi)/2$
by an even permutation of coordinates and all possible sign combinations, where $\varphi=(1+\sqrt 5)/2$ is the golden ratio.

\begin{corollary}
\label{admcor}
If $(x_1,x_2,x_3)\in\A^3(\C)$ is an admissible triple with finite $\Mod(\Sigma)$-orbit, then it corresponds to a representation $\rho:\pi_1(\Sigma)\to \SL_2(\C)$ with  finite image.
\end{corollary}

\begin{proof}[Proof of the corollary]
Replacing $(x_1,x_2,x_3)$ by another triple within its $\Mod(\Sigma)$-orbit if necessary, we may assume that $(x_1,x_2,x_3)$ is one of the triples in Theorem \ref{dubm} or its image under one of the transformations $\sigma_{12}$, $\sigma_{23}$, or $\sigma_{13}$. We shall show that
$$(x_1,x_2,x_3)=(\tr A,\tr B,\tr(AB))$$
where $A,B\in\SL_2(\C)$ are elements that together lie in one of the finite subgroups $B A_4$, $BS_4$, or $BA_5$ of $\SL_2(\C)$. Since the matrix $-\mathbf 1$ is contained in every one of these groups, it suffices to treat the case where $(x_1,x_2,x_3)$ is one of the triples in Theorem \ref{dubm}. By explicit computation, we find that the triples in Theorem \ref{dubm} respectively correspond to traces of the triples of matrices
\begin{align*}
&\left(\begin{bmatrix}0 & 1\\ -1 & 0\end{bmatrix},\begin{bmatrix}-\frac{1}{2}(1+i) & \frac{1}{2}(1-i)\\ -\frac{1}{2}(1+i) & -\frac{1}{2}(1-i)\end{bmatrix},\begin{bmatrix}-\frac{1}{2}(1+i) & -\frac{1}{2}(1-i)\\ \frac{1}{2}(1+i) & -\frac{1}{2}(1-i)\end{bmatrix}\right),\\
&\left(\begin{bmatrix}0 & \frac{1}{\sqrt{2}}(1-i)\\ \frac{1}{\sqrt 2}(1+i) & 0\end{bmatrix},\begin{bmatrix}-\frac{1}{2}(1+i) & \frac{1}{2}(1-i)\\ -\frac{1}{2}(1+i) & -\frac{1}{2}(1-i)\end{bmatrix},\begin{bmatrix}-\frac{1}{\sqrt 2} & \frac{1}{\sqrt 2}i\\ \frac{1}{\sqrt 2 }i & -\frac{1}{\sqrt 2}\end{bmatrix}\right),\\
&\left(\begin{bmatrix}0 & 1\\ -1 & 0\end{bmatrix},\begin{bmatrix}-\frac{1}{2} & \frac{1}{2}(\varphi+\varphi^{-1}i)\\ -\frac{1}{2}(\varphi-\varphi^{-1}i) & -\frac{1}{2}\end{bmatrix},\begin{bmatrix}-\frac{1}{2}(\varphi-\varphi^{-1}i) & -\frac{1}{2}\\ \frac{1}{2} & -\frac{1}{2}(\varphi+\varphi^{-1}i)\end{bmatrix}\right),\\
&\left(\begin{bmatrix}0 & 1\\ -1 & 0\end{bmatrix},\begin{bmatrix}\frac{1}{2}(1-\varphi i) & \frac{1}{2}\varphi^{-1}\\ -\frac{1}{2}\varphi^{-1} & -\frac{1}{2}(1+\varphi i)\end{bmatrix},\begin{bmatrix}-\frac{1}{2}\varphi^{-1} & -\frac{1}{2}(1+\varphi i)\\ -\frac{1}{2}(1-\varphi i) & -\frac{1}{2}\varphi^{-1}\end{bmatrix}\right),\\
&\left(\begin{bmatrix}0 & 1\\ -1 & 0\end{bmatrix},\begin{bmatrix}-\frac{\varphi}{2} &\frac{\varphi^{-1}}{2}+\frac{1}{2}i\\-\frac{\varphi^{-1}}{2}+\frac{1}{2}i & -\frac{\varphi}{2}\end{bmatrix},\begin{bmatrix}-\frac{\varphi^{-1}}{2}+\frac{1}{2}i & -\frac{\varphi}{2}\\\frac{\varphi}{2} &-\frac{\varphi^{-1}}{2}-\frac{1}{2}i\end{bmatrix}\right),
\end{align*}
where $\varphi=(1+\sqrt 5)/2$ is the golden ratio. The matrices for the first triple all lie in the binary tetrahedral group $BA_4$, the matrices for the second triple all lie in the binary octahedral group $BS_4$, and the matrices for the remaining three triples all lie in the binary icosahedral group $BA_5$. In each triple, the third matrix is the product of the first two. Thus, each of the triples in Theorem \ref{dubm} correspond to representations $\pi_1(\Sigma)\to\SL_2(\C)$ with finite image, proving the corollary.
\end{proof}

\begin{remark}
In \cite{dm}, two proofs of Theorem \ref{dubm} are given. The first proof is based on an explicit analysis of certain relevant trigonometric Diophantine equations. General equations of this type are effectively solvable by Lang's $\G_m$ conjecture (proved by Laurent \cite{laurent}), as noted in \cite{bgs}. The second proof in \cite{dm}, based on a suggestion of Vinberg, uses consideration of certain representations of Coxeter groups of reflections associated to admissible triples. Both methods use special features present in the once-punctured torus case which do not seem to generalize easily to the case of general surfaces treated in our work.
\end{remark}

\section*{Acknowledgments} Theorems A and B along with parts of this paper originally appeared in Chapter 6 of 
JPW's Ph.D.~thesis \cite{whang0}, but the proof had a gap which has been filled in this paper. JPW thanks Peter 
Sarnak and Phillip Griffiths for encouragement, and Anand Patel and Ananth Shankar for collaborative work in 
\cite{psw} which led to one of the key ingredients in this paper. IB and MM thank Bill Goldman for useful email 
correspondence, in particular for alerting us to dihedral representations. We thank the anonymous referee for a 
careful reading and useful comments. IB and MM are 
supported in part by  the Department of Atomic Energy, Government of India, under project no.12-R\&D-TFR-5.01-0500. 
Research of IB partly supported by a DST JC Bose Fellowship. SG 
acknowledges the SERB, DST (Grant no. MT/2017/000706) and the Infosys Foundation for their support. 
MM is   supported in part
 by an endowment of the Infosys Foundation,
a DST JC Bose Fellowship, Matrics research project grant  MTR/2017/000005,
and CEFIPRA  project No. 5801-1. MM was also partially supported by the grant 346300 for IMPAN from the Simons Foundation 
and the matching 2015-2019 Polish MNiSW fund.

\end{document}